\documentclass[english]{amsart}

\usepackage{amsmath, amsthm}
\usepackage{amssymb}
\usepackage{amscd}
\usepackage{latexsym}
\usepackage[french]{babel}
\usepackage[utf8]{inputenc}
\usepackage[shortlabels]{enumitem}
\usepackage{xr}
\input xy
\xyoption{all}

\usepackage{xypic}
\usepackage{graphicx}
\usepackage{changebar,color}

\newtheorem{theorem}{Théorème}[section]
\newtheorem{corollary}[theorem]{Corollaire}
\newtheorem{lemma}[theorem]{Lemme}

\newtheorem{proposition}[theorem]{Proposition}

\newtheorem{question}[theorem]{Question}

\theoremstyle{remark}

\theoremstyle{definition}

\newtheorem{definition}[theorem]{Definition}

\theoremstyle{plain}

\theoremstyle{definition}

\numberwithin{equation}{subsection}

\newcommand{\Ker}{\operatorname{Ker}}

\def\Z{\mathbb{Z}}

\def\Q{\mathbb{Q}}

\def\C{\mathbb{C}}

\newcommand{\tor}{{tors}}

\newcommand{\Ok}{{\mathcal O}_k}

\newcommand{\Spec}{{\rm Spec\, }}

\newcommand{\coker}{{\rm coker\,}}

\newcommand{\Pic}{{\rm Pic}}
\newcommand{\NS}{{\rm NS}}

\newcommand{\ra}{\rightarrow}

\newcommand{\lra}{\longrightarrow}

\title{Finitude uniforme pour les cycles de codimension 2 sur les corps de nombres}

\author{Fran\c{c}ois Charles et Alena Pirutka}
\address{\hspace{-0.5cm} Fran\c{c}ois Charles, 
DMA,\newline
\'Ecole Normale Sup\'erieure de Paris,\newline 
Paris, France\newline
Laboratoire de Math\'ematiques d'Orsay, \newline
Universit\'e Paris-Saclay, 
Orsay, France\newline 
\textit{francois.charles@ens.fr}\newline
\newline
\hspace{-0.5cm} Alena Pirutka, 
Courant Institute of Mathematical Sciences, \newline
New York University, 
New York, U.S.A.\newline
\textit{pirutka@cims.nyu.edu}}

\begin{document}

\maketitle

\begin{abstract}

\vspace{-1cm}

Soit $X$ une vari\'et\'e projective et lisse, d\'efinie sur un corps de nombres. Sous l'hypoth\`ese $H^2(X,\mathcal O_X)=0,$ Colliot-Th\'el\`ene et Raskind ont d\'emontr\'e que le sous-groupe de torsion $CH^2(X)_{tors}$ du groupe de Chow en codimension $2$ est fini. Dans cette note, on donne des bornes uniformes pour le groupe fini $CH^2(X)_{tors}$ quand $X$ varie en famille. 

\

\noindent \textsc{Abstract.} Let $X$ be a smooth projective variety defined over a number field. Assuming $H^2(X,\mathcal O_X)=0,$ Colliot-Th\'el\`ene and Raskind proved that the torsion subgroup $CH^2(X)_{tors}$ in the Chow group of cycles of codimension $2$ is finite. In this note, we give uniform bounds for the finite group $CH^2(X)_{tors}$ when $X$ varies in a family.

\end{abstract}

\section{Introduction}

\subsection{Contexte}

Soit $k$ un corps de nombres. Si $E$ est une courbe elliptique sur $k$, le théorème de Mordell-Weil garantit que le sous-groupe $E(k)_{tors}$ des points de torsion de $E$ est fini. D'importants résultats permettent d'exhiber des bornes pour la taille de ce groupe de torsion. Si $k=\mathbb Q$, un th\'eor\`eme de B. Mazur \cite{Mazur} donne une liste finie de toutes les possibilit\'es pour le groupe $E(\mathbb Q)_{tors}.$ les travaux de S. Kamienny, M.A. Kenku, et F. Momose \cite{Ka92,KM88} étendent le théorème de Mazur au cas où $k$ est un corps quadratique. Dans le cas g\'en\'eral, un th\'eor\`eme de L. Merel \cite{Merel} donne une borne uniforme $B(d)$ qui ne d\'epend que du degr\'e $d$ de $k$ sur $\mathbb Q$, telle que $|E(k)_{tors}|\leq B(d)$.

Soit maintenant $C$ une courbe projective lisse sur $k$ de genre $g\geq 2$; l'ensemble $C(k)$ est fini d'apr\`es la conjecture de Mordell, démontrée par Faltings. Des travaux récents de V. Dimitrov, Z. Gao, et Ph. Habegger \cite{DGH} permettent de donner une borne uniforme pour le cardinal $|C(k)|$ qui ne d\'epend que des invariants $g,d,$ et du rang du groupe des points rationnels de la Jacobienne de $C$.

En dimension sup\'erieure, soit $A$ une vari\'et\'e ab\'elienne sur $k$ de dimension au moins $1$. Le groupe de torsion $A(k)_{tors}$ est fini, et on conjecture l'existence d'une borne uniforme pour la taille de ce groupe, qui ne d\'epend que de la dimension de $A$ et du degr\'e $d$ de $k$ sur $\mathbb Q$. Cette conjecture est largement ouverte en g\'en\'erale. Un th\'eor\`eme d'A. Cadoret et A. Tamagawa \cite{CTU,CT} donne un r\'esultat partiel dans cette direction pour la torsion $\ell$-primaire dans les familles \`a un param\`etre: si $S$ est une courbe sur $k$, $\mathcal A\to S$ est une famille de vari\'et\'es ab\'eliennes sur $S$, et $\ell$ est nombre premier, alors pour tout $s\in S(k)$ le sous-groupe de torsion $\ell$-primaire  $\mathcal A_s(k)\{\ell\}$ est d'ordre born\'e par une constante qui ne d\'epend que de $\mathcal A,\,d$ et de $\ell$, mais pas de choix du point $s$.  Pour obtenir l'uniformit\'e de tout le groupe de torsion  dans ce cas, il resterait donc \`a borner l'exposant du groupe $\mathcal A_s(k)_{tors}$: on ne sait pas le faire en g\'en\'eral, m\^eme pour une famille sur une courbe.

Le groupe $A(k)_{tors}$ s'identifie au groupe de Picard de la variété abélienne duale de $A$. Le but de cet article est d'exposer quelques résultats de finitude uniforme pour les cycles de codimension supérieure. Comme indiqué plus haut, ces questions d'uniformité apparaissent sous plusieurs formes dans la littérature, voir par exemple \cite[VI.45.7.3]{Ma86}

Soit $X$ une vari\'et\'e projective lisse sur $k$. Soit $i>0$ et soit $CH^i(X)$ le groupe de Chow des cycles de codimension $i$ sur $X$. Une version forte de la conjecture de Bass pr\'edit que le groupe $CH^i(X)$ est de type fini; son sous-groupe de torsion $CH^i(X)_{tors}$ serait donc un groupe fini. J.-L. Colliot-Th\'el\`ene, W. Raskind \cite{CTR}, et P. Salberger \cite{S91} ont \'etabli la conjecture de finitude de la torsion pour les cycles de codimension $2$ dans le cas suivant:
\begin{theorem} (\cite{CTR}, \cite{S91})\label{theoremeCTR}
Soit $X$ une vari\'et\'e projective  lisse g\'eom\'etriquement int\`egre sur un corps de nombres. Supposons que $H^2(X,\mathcal O_X)=0$. Alors le groupe $CH^2(X)_{tors}$ est fini.
\end{theorem}
En dehors de ce r\'esultat, on ne connait pas d'autres cas g\'en\'eraux o\`u l'on sait \'etablir la finitude des groupes $CH^i(X)_{tors}$ pour les cycles de codimension $i\geq 2$ (voir cependant \cite{L1},\cite{L-R},\cite{L-S},\cite{O1} pour des cas particuliers).

\subsection{L'exemple des surfaces de Châtelet}\label{exempleChatelet}

On donne un exemple o\`u le calcul explicite de la torsion dans $CH^2$ montre que l'on ne peut pas esp\'erer borner uniformément la torsion sans imposer de condition suppl\'ementaire (voir \cite[Corollaire 2 p. 207 et Proposition 13 p. 203]{CTS77} pour un exemple semblable dans le cas des compactifications lisses de tores algébriques, et voir aussi \cite[Theorem 8.13, Remark 8.13.1, Remark 8.13.2]{CTS87} et \cite{Siksek12} pour des résultats négatifs analogues portant sur des questions de points sur les surfaces cubiques). 

Soit $d\in \mathbb Q$ sans facteur carré et soit $\pi:\mathcal X\to \mathcal S$ la famille projective et lisse de surfaces de Ch\^atelet donn\'ee par l'\'equation 
\begin{equation}\label{chatelet}
y^2-dz^2=x(x-t_1)(x-t_2),
\end{equation} sur un ouvert affine. Ici  $t_1, t_2$ sont des coordonn\'ees de $\mathbb A^2_{\mathbb Z}$ et $\mathcal S\subset \mathbb A^2_{\mathbb Z}$ est l'ouvert de lissit\'e de la famille (\ref{chatelet}). 
Notons que l'hypoth\`ese  $H^2(\mathcal X_s,\mathcal O_{\mathcal X_s})=0$ est satisfaite car les fibres de $\pi$ sont des vari\'et\'es rationnellement connexes.

On prend $k=\mathbb Q$, $s\in \mathcal S(k)$ et on pose $X=\mathcal X_s$.  Alors on sait calculer $CH^2(X)_{tors}$. En effet, 
d'apr\`es \cite[Thm.8.13]{CTS87}, \cite{D00}, \cite{D06} le groupe $CH^2(X)_{tors}$ co\"incide avec le groupe $A_0(X)$ des z\'ero-cycles sur $X$ de degr\'e $0$.  D'apr\`es une conjecture de J.-L. Colliot-Th\'el\`ene et J.-J. Sansuc (\cite[Conjecture A]{CTS81}, \cite{CTS87}, \cite{HW}) pour $X$, on a une suite exacte (voir \cite[IV, p.231]{CTS80}  pour l'injectivit\'e de la premi\`ere fl\`eche)
\begin{equation*}
0\to A_0(X)\to \oplus_{v} A_0(X_{k_v})\to H^1(k, \hat S)
\end{equation*}
o\`u la somme de milieu est une somme sur toutes les places $v$ de $k$ et le groupe de droite s'identifie \`a $(\mathbb Z/2)^3$. 
Les groupes $A_0(X_{k_v})$ sont calcul\'es dans \cite{D00}. Ici on a besoin du cas particulier suivant: on a que $A_0(X_{k_v})\simeq (\mathbb Z/2)^2$ sous conditions
que  l'extension $k_v(\sqrt d)/k_v$ est non-ramifi\'ee, et que
$v(t_1)=v(t_2)$ est un entier positif impair.

Ainsi, si  $s=(n_1, n_2)\in \mathcal S(\mathbb Q)$, o\`u $n_1, n_2$ sont deux entiers impairs sans facteur carr\'e, premiers \`a $d$, avec $m$  diviseurs premiers communs, on a $|CH^2(\mathcal X_{s})_{tors}|\geq 2^{2m-3}$. 

\subsection{\'Enonc\'e des r\'esultats}

On se place  dans le cadre suivant. Soit $\mathcal S$ un sch\'ema g\'eom\'etriquement int\`egre de type fini sur Spec$\,\mathbb Z$, soit $S$ la fibre g\'en\'erique de $\mathcal S$ sur $\mathbb Q$,  soit $\pi:\mathcal X\to \mathcal S$ un morphisme projectif et lisse, et soit $\bar{\eta}$ le point g\'en\'erique g\'eom\'etrique de $\mathcal S$. On suppose  $H^2(\mathcal X_s,\mathcal O_{\mathcal X_s})=0$ pour tout point $s\in \mathcal S$. Notons que si le corps r\'esiduel $k$ de $s$ est de caract\'eristique z\'ero, alors cette condition est impliqu\'ee par la condition   $b_2(\mathcal X_
{\bar \eta})=\rho(\mathcal X_{\bar \eta}).$  

On s'int\'eresse  \`a trouver des bornes uniformes pour les groupes finis $CH^2(\mathcal X_s)_{tors}$ qui d\'ependent de $\pi$, de degr\'e de $k$ sur $\mathbb Q$, mais pas de choix du point $s$.  En effet, les m\'ethodes d\'ev\'elopp\'ees dans \cite{CTR} permettent de relier le groupe de Chow des cycles de codimension $2$ avec des invariants de nature cohomologique (cohomologie \'etale, cohomologie des faisceaux $\mathcal K_1$ et $\mathcal K_2$ venant de la $K$-th\'eorie, et autres): on peut donc esp\'erer contr\^oler ces invariants en famille.

Soient $\bar k$ une cl\^oture alg\'ebrique de $k$ et  $\bar s$ le $\bar k$-point de $\mathcal S$ correspondant. Soit $G_k=Gal(\bar k/k)$ le groupe de Galois absolu de $k$. La donn\'ee de la famille arithm\'etique sur $\mathcal S$ permet naturellement de consid\'erer les groupes suivants : 
\begin{enumerate}
\item l'image $\mathrm{Im}[CH^2(\mathcal X_s)_{tors}\to CH^2(\mathcal X_{\bar s})^{G_k}]$;
\item le noyau $\Ker[CH^2(\mathcal X_s)_{tors}\to CH^2(\mathcal X_{\bar s})^{G_k}]$;
\item l'image $\mathrm{Im}[CH^2(\mathcal X_U)_{tors}\to CH^2(\mathcal X_{s})]$ si 
$s : \Spec k\lra \mathcal S$
s'étend en un morphisme $U\lra \mathcal S$ o\`u  $\Ok$ est l'anneau des entiers de $k$, et o\`u $U\subset \Spec\Ok$ est un ouvert.
\end{enumerate}

Nous appellerons le premier de ces groupes la \emph{partie géométrique} de $CH^2(X)_{tors}$. Le second est la \emph{partie arithmétique}.

On \'etudie la premiere partie g\'eom\'etrique dans la section \ref{sectiongeo} (voir Th\'eor\`eme \ref{theorem:exponentgeo} et Proposition \ref{geomgen}) -- comme il est habituel, c'est le théorème de Merkurjev-Suslin qui en permet l'étude cohomologique. Quand la base $S$ est une courbe, les m\'ethodes d'A. Cadoret et A. Tamagawa \cite{CTU, CT} permettent d'obtenir des r\'esultats uniformes pour la partie $\ell$-primaire. Dans la section \ref{sectionarith} on utilise les arguments de J.-L. Colliot-Th\'el\`ene et W. Raskind, et P. Salberger pour trouver des bornes uniformes pour l'exposant du deuxième groupe ci-dessus (voir Th\'eor\`eme \ref{exposant2-k} et corollaire \ref{exposant2}).  En particulier, on obtient :

\begin{theorem}\label{theorem:exposant}
Soit $k$ un corps de nombres, et soit $S$ une courbe quasi-projective, g\'eom\'etriquement int\`egre  sur $k$.  Soit 
$$\pi : \mathcal X\lra S$$ 
un morphisme projectif lisse \`a fibres g\'eom\'etriquement int\`egres. Supposons  que  $H^2(\mathcal X_s,\mathcal O_{\mathcal X_s})=0$ pour tout point $s\in S$. Soit $\ell$ un nombre premier. Pour tout entier strictement positif $d$, il existe un entier $N=N(\pi, \ell, d)$ qui ne d\'epend que de $\pi,\ell$, et $d$,  tel que, pour toute extension finie $K$ de $k$ de degré au plus $d$, et tout $K$-point $s$ de $S$, on a:
$$\ell^N CH^2(\mathcal X_s)\{\ell\}=0.$$
\end{theorem}

Ce dernier résultat motive la question suivante :

\begin{question}
Soit $S$ un schéma de type fini sur $\Q$, et soit $\pi : \mathcal X\ra S$ un morphisme projectif et lisse. Soient $d,i$ deux  entiers strictement positifs. Existe-t-il un entier $N=N(\pi,d,i)$ qui ne d\'epend que de $\pi,d$, et $i$, tel que, pour tout $k$-point $s$ de $S$, o\`u $k$ est un corps de nombres de degr\'e $[k:\mathbb Q]=d$, on ait
$$N CH^i(\mathcal X_s)_{tors}=0?$$
\end{question}

La section \ref{exempleChatelet} montre que, déjà dans le cas des familles de surfaces de Châtelet, on ne peut pas s'attendre à ce que le Théorème \ref{theorem:exposant} s'étende en une bonne uniforme sur la torsion : étant donné $n$, le sous-groupe de $n$-torsion peut n'être pas uniformément borné.

Nous obtenons un tel résultat en fixant des modèles entiers des variétés considérées comme suit. Soit $U$ un ouvert non vide de $\mathrm{Spec}\,\mathcal{O}_k$.

Dans la section \ref{sectionen}, on s'int\'eresse aux $U$-points de $\mathcal S$ avec $U$ fix\'e. On utilise les m\'ethodes de J.-L. Colliot-Th\'el\`ene et W. Raskind,  P. Salberger,  M. Somekawa pour \'etudier le groupe $\mathrm{Im}[CH^2(\mathcal X_U)_{tors}\to CH^2(\mathcal X_{s})]$. On déduit de cette étude le résultat suivant :

\begin{theorem}\label{theorem:torsion}
Soit $\mathcal S$ un sch\'ema intègre, s\'epar\'e, de type fini sur $\mathbb Z$. Soit 
$\pi : \mathcal X\lra \mathcal S$
un morphisme projectif lisse \`a fibres g\'eom\'etriquement int\`egres. Supposons  que  $H^2(\mathcal X_s,\mathcal O_{\mathcal X_s})=0$ pour tout point $s\in \mathcal S$.
Soit $k$ un corps de nombres et soit $\mathcal O_k$ son anneau des entiers, soit $U\subset \Spec\Ok$ un ouvert non vide. Soit 
$s : \Spec k\lra \mathcal S$ l'image du point g\'en\'erique d'un $U$-point de $\mathcal S$.  Il existe un entier $N=N(\pi, U)$ tel que : 
$$|CH^2(\mathcal X_s)_{tors}|\leq N.$$
\end{theorem}

\subsection{Notations}

Si $A$ est un ensemble fini, on \'ecrit $\# A$ ou $|A|$ pour le nombre des \'el\'ements de $A$. Si $A$ est un groupe ab\'elien, $n>0$ est un entier positif et $\ell$ est un nombre premier, on d\'efinit $A[n]=\{x\in A \,|\,nx=0\}$ le sous-groupe des \'el\'ements de $n$-torsion, $A\{\ell\}=\{x\in A \,|\exists n\;\ell^nx=0\}$ le sous-groupe des \'el\'ements de torsion $\ell$-primaire, et $A_{tors}=\{x\in A \,|\exists n,\;nx=0\}$ le sous-groupe des \'el\'ements de torsion.

Soit $X$ un sch\'ema noeth\'erien. Si $n$ est un entier inversible sur $X$, on note $\mu_{n}$ le faisceau \'etale sur $X$ d\'efini par les racines $n$-i\`emes de l'unit\'e. Pour $j$ un entier positif, on note $\mu_{n}^{\otimes j}=\mu_{n}\otimes\ldots\otimes\mu_{n}$ ($j$ fois).  On pose $\mu_{n}^{\otimes j}=Hom(\mu_{n}^{\otimes (-j)}, \mathbb Z/n)$ si $j$ est n\'egatif et $\mu_{n}^{\otimes 0}=\mathbb Z/n$. On note $H^i_{\acute{e}t}(X,\mu_n^{\otimes j})$ les groupes de cohomologie \'etale de $X$ \`a valeurs dans $\mu_n^{\otimes j}$.
Les groupes 
$H^i_{\acute{e}t}(X,\mathbb Q_{\ell}/\mathbb Z_{\ell}(j))$, resp.
 $H^{i}_{\acute{e}t}(X,\mathbb Z_{\ell}(j))$ sont obtenus par passage \`a la limite inductive, resp. projective dans les groupes $H^i_{\acute{e}t}(X,\mu_n^{\otimes j})$ lorsque $n$ varie  parmi les puissances d'un nombre premier $\ell$ inversible sur $X$. On \'ecrit $\mathbb G_m$ pour le groupe multiplicatif et le faisceau \'etale ainsi d\'efini sur le sch\'ema $X$; on a Pic$(X)\simeq H^1_{\acute{e}t}(X,\mathbb G_m)$.
  
Si $i$ est un entier positif, on note $X^{(i)}$ l'ensemble des points de $X$ de codimension $i$ et on note $CH^i(X)$ le groupe des cycles de codimension $i$ modulo l'\'equivalence rationnelle \cite{Fulton}.
Si $j$ est un entier positif, on note $\mathcal K_j$ le faisceau de Zariski associ\'e au pr\'efaisceau $U\mapsto K_j(H^0(U, \mathcal O_U))$, le groupe $K_j(A)$ \'etant celui associ\'e par Quillen \`a l'anneau $A$.

Si $k$ est un corps, on \'ecrit $\bar k$ pour d\'esigner une cl\^oture alg\'ebrique de $k$, on \'ecrit $k^s$ pour une cl\^oture s\'eparable de $k$, $G_k=Gal(k^s/k)$ est le groupe de Galois absolu.  Si $X$ est une vari\'et\'e alg\'ebrique d\'efinie sur un corps $k$, on note $\bar X=X_{\bar k} =X\times_k \bar{k}$. Si $X$ est int\`egre, on note $k(X)$ son corps des fonctions. 

Si $k$ est un corps de nombres, on \'ecrit $\Ok$ pour l'anneau des entiers de $k$. Si $U\subset \Spec \Ok$ est un ouvert, les places {\it \`a l'infini} de $U$ sont les places finies de $\Ok$ qui  correspondent aux id\'eaux premiers du compl\'ementaire de $U$ dans $\Spec \Ok$.

Si $X$ est une vari\'et\'e projective  lisse sur un corps $k$ s\'eparablement clos, on \'ecrit $b_2(X)=\rho(X)$ si pour tout premier $\ell$ diff\'erent de la caract\'eristique de $k$,   l'application naturelle $\mathrm{Pic}(X)\otimes \mathbb Q_{\ell}\to H^2_{\acute{e}t}(X, \mathbb Q_{\ell}(1))$ est surjective. Notons qu'il suffit de le v\'erifier pour un seul premier $\ell$.  Cette propri\'et\'e est stable par sp\'ecialisation.  Si $k$ est de caract\'eristique nulle, le th\'eor\`eme de Lefschetz sur les classes de type $(1,1)$ montre que $b_2(X)=\rho(X)$ si et seulement si $H^2(X, \mathcal O_X)~=~0$. 

\

\noindent\textbf{Remerciements.} Le premier auteur remercie l'Institut Courant, NYU pour son hospitalit\'e. Il est soutenu par le projet ERC AlgTateGro (Horizon 2020 Research and Innovation Programme, grant agreement No 715747).

Les auteurs remercient Anna Cadoret pour plusieurs discussions et commentaires sur le manuscrit, ainsi que Jean-Louis Colliot-Thélène pour plusieurs discussions.

\section{Uniformité de l'exposant}

Soit $k$ un corps de nombres et soit $X$ une vari\'et\'e projective et lisse sur $k$. Dans cette section on s'int\'eresse \`a l'exposant du groupe $CH^2(X)_{tors}$. Le but est de démontrer le théorème \ref{theorem:exposant}. On commence par borner l'exposant de la partie g\'eom\'etrique de ce groupe, \`a savoir son image dans le groupe $CH^2(\bar X)$; on \'etudie ce dernier groupe via l'application d'Abel-Jacobi $\ell$-adique.

\subsection{Image de $CH^i(X)_\tor$ par l'application d'Abel-Jacobi}\label{sectiongeo}

\subsubsection{} Soit $S$ une courbe quasi-projective sur $k$, et soit $\pi:\mathcal X\to S$ un morphisme projectif lisse. Soit $\ell$ un nombre premier. Soit $\overline\eta$ un point générique géométrique de $S$. 

Soit $K$ une extension du corps $k$. Si $s$ est un $K$-point de $S$, et $\overline s$ est un point géométrique de $S$ au-dessus de $s$, le groupe de Galois absolu $G_K$ de $K$ agit naturellement sur les groupes de cohomologie étale 
$$H^i_{\acute{e}t}(\mathcal X_{\overline s}, \Q_\ell/\Z_\ell(j))$$
où $i$ est un entier positif et $j$ un entier arbitraire.

Le résultat suivant suit du théorème principal de \cite{CT}. 

\begin{proposition}\label{proposition:fini-coh}
Soit $d$ un entier strictement positif. Soient $i$ et $j$ deux entiers avec $i\geq 0$  et $i\neq 2j$. Il existe un entier $N=N(\pi,i,j,d)$ tel que, pour toute extension finite $K$ de $k$ de degré au plus $d$, tout $K$-point $s$ de $S$, et tout point géométrique $\overline s$ de $S$ au-dessus de $s$, on ait:
$$|H^i_{\acute{e}t}(\mathcal X_{\overline s}, \Q_\ell/\Z_\ell(j))^{G_K}|\leq \ell^N.$$
\end{proposition}

\begin{proof}
On applique \cite[Corollary 4.2]{CT} à la représentation du groupe fondamental étale $\pi^1(S, \overline\eta)$ de $S$ sur $M=H^{i}_{\acute{e}t}(X_{\bar\eta}, \Z_\ell(j))$, en prenant pour $\chi$ le caractère trivial -- les conjectures de Weil assurant l'hypothèse sur $\chi$. On en déduit un entier $N_1$ tel que, pour toute extension $K$ de $k$ de degré au plus $d$ et tout $K$-point $s$ de $S$, on ait
$$(H^{i}_{\acute{e}t}(\mathcal X_{\overline s}, \Z_\ell(j))\otimes\Q_\ell/\Z_\ell)^{G_K}\subset H^{i}_{\acute{e}t}(\mathcal X_{\overline s}, \Z_\ell(j))/\ell^{N_1}.$$

Considérons par ailleurs la suite exacte $G_K$-équivariante (voir par exemple \cite[p.781]{CTSS}):
$$0\lra H^{i}_{\acute{e}t}(\mathcal X_{\overline s}, \Z_\ell(j))\otimes\Q_\ell/\Z_\ell\lra H^{i}_{\acute{e}t}(\mathcal X_{\overline s}, \Q_\ell/\Z_\ell(j))\lra H^{i+1}_{\acute{e}t}(\mathcal X_{\overline s}, \Z_\ell(j))_\tor\lra 0.$$
On en déduit une suite exacte
$$0\lra (H^{i}_{\acute{e}t}(\mathcal X_{\overline s}, \Z_\ell(j))\otimes\Q_\ell/\Z_\ell)^{G_K}\lra H^{i}_{\acute{e}t}(\mathcal X_{\overline s}, \Q_\ell/\Z_\ell(j))^{G_K}\lra H^{i+1}_{\acute{e}t}(\mathcal X_{\overline s}, \Z_\ell(j))_\tor.$$
Le terme de droite étant un groupe fini dont la classe d'isomorphisme est indépendante de $s$, cela conclut la preuve.
\end{proof}

\subsubsection{Finitude pour l'application d'Abel-Jacobi de Bloch et conséquence en codimension $2$}

Soit $X$ une variété projective lisse sur un corps $k$. Soit $\ell$ un nombre premier inversible dans $k$. 

Dans \cite{Bloch}, Bloch définit une application d'Abel-Jacobi, fonctorielle en $X$ pour l'action des correspondances:
$$AJ^i_\ell : CH^i(X_{\overline k})\{\ell\}\lra H^{2i-1}_{\acute{e}t}(X_{\overline k}, \Q_\ell/\Z_\ell(i)).$$
Par fonctorialité on obtient une application, que nous noterons aussi $AJ^i_\ell$ : 
$$AJ^i_\ell : CH^i(X)\{\ell\}\lra H^{2i-1}_{\acute{e}t}(X_{\overline k}, \Q_\ell/\Z_\ell(i))^{G_k}.$$

Si $k$ est algébriquement clos, l'application $AJ^d$ est un isomorphisme et l'application $AJ^2$ est injective (voir \cite{Bloch}, \cite[Th\'eor\`eme 4.3]{CT-cycles}). On déduit immédiatement de la Proposition \ref{proposition:fini-coh} l'énoncé suivant:
\begin{theorem}\label{theorem:exponentgeo}
Soit $S$ une courbe quasi-projective sur un corps de nombres $k$, et soit $\pi:\mathcal X\to S$ un morphisme projectif lisse. Soit $\ell$ un nombre premier. 

Soit $d$ un entier strictement positif. Il existe un entier $N=N(\pi,d)$ tel que, pour toute extension finie $K$ de $k$ de degré au plus $d$, tout $K$-point $s$ de $S$, et tout point géométrique $\overline s$ de $S$ au-dessus de $s$, on ait:
$$|\mathrm{Im}(CH_0(\mathcal X_s)\lra CH_0(\mathcal X_{\overline s}))\{\ell\}|\leq \ell^N$$
et:
$$|\mathrm{Im}(CH^2(\mathcal X_s)\lra CH^2(\mathcal X_{\overline s}))\{\ell\}|\leq \ell^N.$$
\end{theorem}

%\begin{theorem}\label{exposantgeom}
%Soit $B$ une courbe quasi-projective sur $\mathbb Q$ et soit $\pi:\mathcal X\to B$ un morphisme projectif lisse, \`a fibres g\'eometriquement connexes; soit $\bar \eta$ le point g\'en\'erique g\'eom\'etrique de $B$. 

%Soit $d>0$ un entier, et soit $\ell$ un nombre premier. Il existe une constante $N$, qui ne d\'epend que de $d$ et de $\pi$ et v\'erifie la propri\'et\'e suivante: 
% soit $k$ un corps de nombres de degré au plus $d$ sur $\Q$ et soit $b$ un $k$-point de $B$, alors le groupe
%$$\mathrm{Im}[ CH^2(\mathcal X_b)\to CH^2(\bar{\mathcal X}_b)][\ell^\infty]=0$$
%est isomorphe à un sous-groupe de $(\Z/\ell^N\Z)^{b_3(\mathcal X_{\overline \eta})}.$ 
%\end{theorem}

La remarque suivante permet, sous des hypoth\`eses plus fortes, de traiter le cas du groupe de torsion tout entier et d'un sch\'ema de base de dimension arbitraire.

%Si l'on consid\`ere tout le sous-groupe de torsion, ou si le sch\'ema de base $S$ est de dimension sup\'erieure, on peut aussi trouver une borne uniforme pour la partie g\'eom\'etrique: cette borne d\'epend aussi d'un nombre fini de r\'eductions sur un corps fini.

\begin{proposition}\label{geomgen}
Soit $\mathcal S$ un sch\'ema s\'epar\'e int\`egre de type fini sur $\mathbb Z$ et soit $\pi:\mathcal X\to \mathcal S$ un morphisme projectif lisse.
Soient $d,r>0$ des entiers. Il existe une constante $N=N(\pi,d,r)$, qui v\'erifie la propri\'et\'e suivante:  soit $k$ un corps de nombres de degr\'e au plus $d$, soient $s$ un $k$-point de $\mathcal S$ et  $\overline s$ un point géométrique de $\mathcal S$ au-dessus de $s$; si le morphisme 
$$s : \Spec k\lra \mathcal S$$
s'étend en un morphisme 
$$U\lra \mathcal S$$
au-dessus d'un ouvert $U\subset \Spec \mathcal O_k$ tel que  $U$ admet au plus $r$ places \`a l'infini, alors 
$$
|\mathrm{Im}(CH^2(\mathcal X_s)\lra CH^2(\mathcal X_{\overline s}))_{tors}|\leq N
$$
\end{proposition}
\begin{proof}
Puisque l'application $AJ^2_{\ell}$ est injective, il suffit de borner la taille du groupe de cohomologie \'etale $H^3_{\acute{e}t}( \mathcal X_{\bar s}, \mathbb Q/\mathbb Z(2))^{G_k}$.  Soit $s_0$  un point ferm\'e de $U$  de corps r\'esiduel un corps fini $\mathbb F$ de caract\'eristique $p$.  On a alors l'inclusion de groupes finis :
\begin{equation}\label{majorationf}
\oplus_{\ell\neq p}H^3_{\acute{e}t}({\mathcal X}_{\bar s}, \mathbb Q_{\ell}/\mathbb Z_{\ell}(2))^{G_k}\subset \oplus_{\ell\neq p} H^3_{\acute{e}t}({\mathcal X}_{\bar s_0}, \mathbb Q_{\ell}/\mathbb Z_{\ell}(2))^{G_{\mathbb F}}
\end{equation}
 (voir \cite[Theorem 5]{CTRK2}). On applique le lemme \ref{places} ci-dessous: puisque le degr\'e de $k$ est born\'e par $d$, et le nombre de places \`a l'infini de $U$ est born\'e par  $r$, on peut choisir un nombre fini de fibres $\mathcal X_{s_0}$ qui donnent donc une borne pour la taille du groupe $H^3_{\acute{e}t}( \mathcal X_{\bar s}, \mathbb Q/\mathbb Z(2))^{G_k}$ d'apr\`es l'inclusion (\ref{majorationf}). 
\end{proof}

\begin{lemma}\label{places}
Soit $\mathcal S$ un sch\'ema s\'epar\'e de type fini sur $\mathbb Z$. Soient $n,d,r>0$ des entiers. Il existe une constante $m$, qui ne d\'epend que de $n,d,r>0$ et $m$ points $s_1,\ldots, s_{m}\in \mathcal S$ tels que:
\begin{enumerate}
\item les corps r\'esiduels $\kappa(s_1),\ldots \kappa(s_{m})$  sont finis, de caract\'eristiques   premi\`eres \`a $n$;
\item si $k$ est un corps de nombres de degr\'e au plus $d$,  si  $U\subset \Spec \Ok$ est un ouvert tel que $U$ admet au plus $r$ places \`a l'infini, et si  $U\to \mathcal S$ est un $U$-point de $\mathcal S$, alors il existe $1\leq i<j\leq m$ tels que les caract\'eristiques de corps r\'esiduels de $\kappa(s_i)$ et $\kappa(s_j)$ sont diff\'erentes, et  les fibres de $U$ au-dessus de $s_i$ et $s_j$ sont non vides.
\end{enumerate}
\end{lemma}
\begin{proof}
%Soit  $T\subset Spec\, \mathbb Z$ un ensemble  fini de points tels que pour tout $p\notin T$ la fibre $B_p$ est nonvide. 
On choisit un ensemble fini $P\subset \mathbb Z$ de premiers qui ne divisent pas $n$, tel que %$P\cap T=\emptyset$ et 
$\#P\geq r+2$. On prend pour $s_1,\ldots, s_{m}$ tous les points ferm\'es de $\mathcal S_p$ pour $p\in P$ de degr\'e au plus $d$ sur $\mathbb F_p$. Notons qu'on n'a qu'un nombre fini de tels points.
 
 \noindent Soit $U$  comme dans l'\'enonc\'e du lemme. Puisque le nombre de places \`a l'infini de $U$ est au plus $r$, et puisque  $\#P\geq r+2$, il existent deux premiers $p_1, p_2 \in P$ tels que les fibres de $U$ au-dessus de $p_1$ et $p_2$ sont non vides, et elles correspondent donc aux points $s_i\in \mathcal S$ au-dessus de $p_1$ et $s_j\in \mathcal S$ au-dessus de $p_2$, pour  certains $1\leq i\neq j\leq m$,  puisque $[k:\mathbb Q]\leq d$.
\end{proof}

\subsection{Exposant de la partie arithmétique}\label{sectionarith}

\subsubsection{Cycles de codimension $2$ et $\mathcal K_2$-cohomologie : rappels}

On rappelle d'abord quelques r\'esultats g\'en\'eraux sur la torsion dans le groupe de Chow des cycles de codimension $2$ et la structure des groupes de $\mathcal K_2$-cohomologie.

\begin{proposition}\label{prch2} 
Soit $X$ une vari\'et\'e projective lisse g\'eometriquement int\`egre sur un corps $k$ de caract\'eristique $0$. Alors il existe des suites exactes :
\begin{enumerate}

\item \scalebox{0.95}{
$0\to H^1(X, \mathcal K_2)\otimes \mathbb Q/\mathbb Z\to NH^3_{\acute{e}t}(X, \mathbb Q/\mathbb Z(2))\to CH^2(X)_{tors}\to 0$,}\\
o\`u $NH^3_{\acute{e}t}(X, \mathbb Q/\mathbb Z(2))=\mathrm{Ker}[H^3_{\acute{e}t}(X, \mathbb Q/\mathbb Z(2))\to  H^3(k(X),  \mathbb Q/\mathbb Z(2))]$;
 \item \scalebox{0.91}{$0\to H^1(X, \mathcal K_2)\otimes \mathbb Q/\mathbb Z \to \Ker\,\tau\to \Ker[CH^2(X)\to CH^2(\bar X)^{G_k}]\to 0$,}\
 
\noindent o\`u $\tau$ est la fl\`eche $\tau:H^3_{\acute{e}t}(X, \mathbb Q/\mathbb Z(2))\to [H^3_{\acute{e}t}(\bar X, \mathbb Q/\mathbb Z(2))^{G_k}\oplus H^3(k(X),  \mathbb Q/\mathbb Z(2))].$

\item \scalebox{0.95}{$0\to \varinjlim_n (\Pic^0_{X/k}(\bar k)[n]\otimes \mu_n) \to H^1_{\acute{e}t}(\bar X, \Q/\Z(2))\to  F_X \to 0$,}\

\noindent o\`u le groupe $F_X$ de droite est fini, de m\^eme ordre que le groupe $\oplus_{\ell} H^2_{\acute{e}t}( \bar X, \mathbb Z_{\ell}(2))\{\ell\}$.
 
\end{enumerate}

Si de plus  $H^2(X,\mathcal O_X)=0$, on a une suite exacte:

\begin{enumerate}\setcounter{enumi}{3}

\item \scalebox{0.95}{$0\to \mathrm{Pic}(\bar X)\otimes \Q/\Z \to H^2_{\acute{e}t}(\bar X, \mathbb Q/\mathbb Z(1)) \to \oplus_{\ell} H^3_{\acute{e}t}(\bar X, \mathbb Z_{\ell}(1))\{\ell\}\to 0$,}\

\noindent et le groupe de droite est fini.
\end{enumerate}
\end{proposition}

\begin{proof}
Pour la suite (1) voir \cite{CT-cycles} (3.11). Pour obtenir la suite (2) on compare les suites (1) pour $X$ et $\bar X$, et on utilise que $H^1(\bar X, \mathcal K_2)\otimes \mathbb Q/\mathbb Z=0$ d'apr\`es \cite[Thm. 2.2]{CTRK2}. La suite (4) est \cite[Proposition 2.11]{CTRK2}.

Pour la suite (3) on utilise l'isomorphisme $H^0(\bar X,\mathcal K_2)_{tors}\simeq H^1_{\acute{e}t}(\bar X, \Q/\Z(2))$ de Suslin (voir \cite[Corollary 5.3]{Suslin}). Ensuite, \cite[Th\'eor\`eme 1.8 et sa preuve]{CTRK2} identifie le sous-groupe divisible maximal du groupe $H^0(\bar X,\mathcal K_2)_{tors}$ avec le groupe de gauche $\varinjlim_n (\Pic^0_{X/k}(\bar k)[n]\otimes \mu_n)$. Le quotient du groupe $H^0(\bar X,\mathcal K_2)_{tors}$ par son sous-groupe divisible maximal est (non canoniquement) isomorphe \`a $\NS(\bar X)_{tors}$, isomorphe lui-m\^eme au groupe $\oplus_{\ell} H^2( \bar X, \mathbb Z_{\ell}(2))\{\ell\}$ (voir \cite[Remarque 1.9, Lemma 1.4]{CTRK2}).

\end{proof}

\subsubsection{Invariants}\label{subsubsection:invariants} 

Soit $k$ un corps et soit $X$ une variété projective lisse g\'eom\'etriquement int\`egre sur $k$. On peut considérer les invariants suivants de la variété $X$:
\begin{enumerate}[(i)]
\item $d_i(X)$ est le degré minimal d'une extension finie $K$ de $k$ telle que l'ensemble des $K$-points $X(K)$ de $X$ est non vide\footnote{Ici cet invariant convient mieux que l'indice de $X$.};
\item $d_{NS}(X)$ est le degr\'e minimal d'une extension finie $K$ de $k$ telle que l'application naturelle $\Pic(X_K) \to \NS(\bar X)$ est surjective, o\`u $\NS(\bar X)$ est le groupe de Néron-Severi de $\bar X$ -- c'est un groupe abélien de type fini; 
\item Si $(i, j)\in \{(2, 1), (2, 2), (3, 2)\}$, $n_{ij}(X)$ est l'ordre du groupe abélien fini $\bigoplus_{\ell}H^i(\overline X, \mathbb Z_\ell(j))_{tors}$.
\end{enumerate}

Notons que les invariants $n_{i,j}(X)$  sont des invariants cohomologiques, constants dans les familles projectives lisses.
Pour contr\^oler la variaition de $d_i(X)$ et $d_{NS}(X)$ en famille, on dispose des deux lemmes faciles ci-dessous:
\begin{lemma}\label{idansf}
 Soit $\mathcal S$ un sch\'ema s\'epar\'e de type fini sur $\mathbb Z$ et soit $\pi:\mathcal X\to \mathcal S$ un morphisme projectif lisse \`a fibres g\'eom\'etriquement int\`egres.
 Il existe une constante $N=N(\pi)$ telle que pour tout point $s\in \mathcal S$, la vari\'et\'e $\mathcal X_{s}$ sur le corps r\'esiduel $\kappa(s)$ de $s$ v\'erifie: $d_i(\mathcal X_{s})\leq N.$
\end{lemma}

\begin{proof}
Soit $\eta$ un point g\'en\'erique de $\mathcal S$, et soit $k(\eta)$ son corps r\'esiduel. Choisissant un point de $\mathcal X_\eta$ à valeurs dans une extension finie $L$ de $k(\eta)$, le morphisme $\Spec L\to \mathcal S$ s'\'etend en un morphisme quasi-fini dominant $f_1: \mathcal S_1\to \mathcal S$, o\`u $\mathcal S_1$ est s\'epar\'e, de type fini sur $\Z$. Par construction, quitte \`a remplacer $\mathcal S_1$ par un ouvert non vide, il existe un morphisme $i_1:\mathcal S_1\to \mathcal X$ tel que $\pi\circ i_1=f_1$. On conclut par r\'ecurrence noeth\'erienne en consid\'erant l'image de $f_1$.
\end{proof}

\begin{lemma}\label{iNSdansf}
Soit $k$ un corps de nombres. Soit $S$ un sch\'ema s\'epar\'e de type fini sur $k$ et soit $\pi:\mathcal X\to S$ un morphisme projectif lisse \`a fibres g\'eom\'etriquement int\`egres, tel que $H^2(\mathcal X_s,\mathcal O_{\mathcal X_s})=0$ pour tout point $s\in S$.
 Il existe une constante $N=N(\pi)$ telle que pour tout point $s\in  S$ on a que  la vari\'et\'e $\mathcal X_{s}$ sur le corps r\'esiduel $\kappa(s)$ de $s$ v\'erifie: $d_{NS}(\mathcal X_{s})\leq N.$
\end{lemma}
\begin{proof}
L'hypoth\`ese d'annulation des groupes $H^2(\mathcal X_s, \mathcal O_{\mathcal X_s})$ garantit que si $\bar\eta$ est un point g\'en\'erique g\'eom\'etrique de $S$, les applications de sp\'ecialisation
$$\NS(\mathcal X_{\bar\eta})\to \NS(\mathcal X_{\bar s})$$
sont des isomorphismes comme on peut le voir en identifiant ces deux groupes aux groupes de cohomologie singuli\`ere de changements de base \`a $\C$ de $\mathcal X_{\overline s}$ et $\mathcal X_{\bar\eta}$ respectivement.

On peut trouver un sch\'ema $S'$ s\'epar\'e, de type fini sur $k$, $S'\to S$ quasi-fini dominant, et des sections $\{l_i\}_i\in \Pic(\mathcal X_{S'})$ qui engendrent le groupe $\NS(\mathcal X_{\bar{\eta}})$. Par sp\'ecialisation, ces sections engendrent les groupes $\NS(\mathcal X_{\bar{s'}})$ pour tout point $s'$ de $S'$. On conclut par r\'ecurrence noeth\'erienne.
\end{proof}

Outre les invariants ci-dessus, on va aussi considérer des invariants obtenus à partir de diviseurs amples dans $X$. Soit $K$ une extension finie de $k$ telle que $X(K)\neq\emptyset$. On se donne une courbe projective lisse et géométriquement intègre $C$ dans $X_K$, avec $C(K)\neq \emptyset$, intersection complète de diviseurs amples. Le théorème de Bertini garantit l'existence d'une telle courbe. On a alors une application injective de vari\'et\'es ab\'eliennes :
$$\Pic^0_{X_K/K}\to \Pic^0_{C/K}.$$ 
Par le th\'eor\`eme de compl\`ete r\'eductibilit\'e de Poincar\'e \cite[Chap. IV, Theorem 1, p.173]{Mumford}, il existe un morphisme $\Pic^0_{C/K}\to \Pic^0_{X_K/K}$ tel que la composition: $$\tau_C: \Pic^0_{X_K/K}\to \Pic^0_{C/K}\to \Pic^0_{X_K/K}$$ est une isog\'enie.

Dans les arguments ci-dessous, on aura besoin de contr\^oler le cardinal du noyau de $\tau_C$, ainsi que l'invariant $n_{22}(C)$ en famille. Pour ce faire, il est commode d'introduire la notation suivante:

\begin{definition}

On \'ecrit
$d_{h,22}(X)$ pour le plus petit entier qui apparaît comme le produit du cardinal du noyau d'une isogénie $\tau_C$ construite comme ci-dessus, et de $n_{22}(C)$. Plus pr\'ecis\'ement, $d_{h,22}(X)$ est le plus petit produit $d_{h,22}(X)=d\cdot n$, tel qu'il existe une extension finie $K/k$ de degr\'e $d_i(X)$ avec $X(K)\neq \emptyset$, une  courbe lisse $C\subset X_K$ obtenue par sections hyperplanes successives, telle que $C(K)\neq \emptyset$, $n_{22}(C)=n$, et telle que la fl\`eche compos\'ee $$\tau_C: \Pic^0_{X_K/K}\to \Pic^0_{C/K}\to \Pic^0_{X_K/K}$$ est une isogénie dont le noyau est de cardinal $d$.

\end{definition}

\begin{lemma}\label{dhdansf}
Soit $k$ un corps de nombres. Soit $S$ un sch\'ema s\'epar\'e de type fini sur $k$ et soit $\pi:\mathcal X\to S$ un morphisme projectif lisse, \`a fibres g\'eometriquement int\`egres. Il existe une constante $N=N(\pi)$ telle que pour tout point ferm\'e $s\in S$   on a: 
$$d_{h,22}(\mathcal X_{s})\leq N.$$
\end{lemma}
\begin{proof}
Comme dans le lemme \ref{idansf}, on va utiliser un argument de r\'ecurrence noeth\'erienne sur le sch\'ema de base $S$. Il suffit donc de trouver une borne $N$ apr\`es un changement de base quasi-fini dominant $S'\to S$. En particulier, on peut supposer que $\pi$ a une section et que $S$ est int\`egre de point g\'en\'erique $\eta$.

Soit $\mathcal C_{\eta}$ une courbe lisse, obtenue par des sections hyperplanes successives de la fibre g\'en\'erique $\mathcal X_{\eta}$. Quitte \`a remplacer $S$ par un sch\'ema quasi-fini sur $S$, on peut supposer que $\mathcal C_{\eta}(\eta)$ est non vide et que $\mathcal C_{\eta}$ s'\'etend en une courbe relative lisse $\mathcal C\subset \mathcal X\to S$ intersection compl\`ete de sections hyperplanes relatives de $\mathcal X$, qui admet une section sur $S$.  L'application $\Pic^0_{\mathcal X_{\eta}}\to \Pic^0_{\mathcal C_{\eta}}$ est injective. D'apr\`es le th\'eor\`eme de compl\`ete r\'eductibilit\'e de Poincar\'e \cite[Chap. IV, Thm.1, p. 173]{Mumford}, il existe une application $\Pic^0_{\mathcal C_{\eta}}\to \Pic^0_{\mathcal X_{\eta}}$ telle que l'application compos\'ee
\begin{equation}\label{isogen} 
\Pic^0_{\mathcal X_{\eta}}\to \Pic^0_{\mathcal C_{\eta}}\to \Pic^0_{\mathcal X_{\eta}}
\end{equation}
est une isog\'enie. Quitte \`a remplacer $S$ par un ouvert non-vide, on peut supposer que la construction ci-dessus s'\'etend sur $S$. Ainsi, en utilisant la compatibilit\'e des sch\'emas $\Pic_{\mathcal X/S}$ et $\Pic_{\mathcal C/S}$ au changement de base, on peut supposer que pour tout point $s$ de $S$, on a une isog\'enie
$$\Pic^0_{\mathcal X_{s}}\to \Pic^0_{\mathcal C_{s}}\to \Pic^0_{\mathcal X_{s}}$$
 dont le cardinal du noyau est born\'e par le cardinal du noyau de l'isog\'enie g\'en\'erique (\ref{isogen}). Puisque  $n_{22}(\mathcal C_s)$ est constant dans la famille $\mathcal C/S$, le lemme est d\'emontr\'e. 
\end{proof}

\subsubsection{Une version effective d'un résultat de Colliot-Th\'el\`ene--Raskind--Salberger}

Soit $X$ une vari\'et\'e projective, lisse, g\'eom\'etriquement int\`egre, d\'efinie sur un corps de nombres $k$, telle que  $H^2(X, \mathcal O_X)=0$.  La finitude du groupe $$Ker[CH^2(X)\to CH^2(\bar X)^{G_k}]$$ est d\'emontr\'ee par J.-L. Colliot-Th\'el\`ene - W. Raskind - P. Salberger (\cite[Th\'eor\`eme 4.3]{CTR}, voir aussi \cite[Th\'eor\`eme 9.1]{CT-cycles}). 

Dans les paragraphes qui suivent, nous reprenons la preuve de \cite{CTR} afin d'en d\'egager une version effective et de donner une borne sur le noyau ci-dessus qui ne d\'epende que des invariants de la section \ref{subsubsection:invariants} ci-dessus. Les arguments ci-dessous sont repris de \cite{CTR} et y apparaissent int\'egralement. Dans le cas   $H^1(X, \mathcal O_X)=H^2(X, \mathcal O_X)=0$, une borne explicite est donn\'ee dans \cite[Th\'eor\`eme A.1]{KCT}.

\begin{theorem}\label{exposant2-k}
Le groupe $Ker[CH^2(X)\to CH^2(\bar X)^{G_k}]$  est d'exposant fini $N$,  born\'e par le degr\'e de $k$ et les invariants décrits en section \ref{subsubsection:invariants}: $d_i(X), d_{NS}(X)$, $n_{ij}(X)$, et $d_{h,22}(X)$.
\end{theorem}

\begin{proof}

Par un argument de restriction-corestriction, quitte \`a multiplier l'exposant $N$  par $d_i(X)$, $d_{NS}(X)$ et $d_{h,22}(X)$, on peut supposer que $X(k)\neq \emptyset$, que $\NS(X)=\NS(\bar X)$, et qu'on a une courbe lisse $C\subset X$ intersection compl\`ete de sections hyperplanes, telle que $C(k)\neq \emptyset$ et telle que la fl\`eche compos\'ee $\Pic^0_{X/k}\to \Pic^0_{C/k}\to \Pic^0_{X/k}$ est une isogenie dont le noyau est de cardinal $d$, et qu'on a $d_{h,22}(X)=dn_{22}(C)$.

\noindent Dans les notations de la Proposition \ref{prch2}(2)  il suffit de montrer que le groupe
$$
I=\mathrm{coker}\,[H^1(X, \mathcal K_2)\otimes \mathbb Q/\mathbb Z\stackrel{\psi}{\to} \\
 \Ker\,\tau]
$$
est d'exposant fini, born\'e par les invariants de la section \ref{subsubsection:invariants}.

\noindent Pour ce faire, J.-L. Colliot-Th\'el\`ene et W. Raskind \cite{CTR}
utilisent la suite spectrale de 
 Hochschild-Serre qui donne une filtration
  sur le groupe $\Ker[H^3_{\acute{e}t}(X, \mathbb Q/\mathbb Z(2))\to H^3_{\acute{e}t}(\bar X, \mathbb Q/\mathbb Z(2))^{G_k}]$: on obtient une suite exacte
  \vspace{-0.2cm}
   
$$
 0\to F_1(X)\to \Ker\, [H^3_{\acute{e}t}(X, \mathbb Q/\mathbb Z(2))\to H^3_{\acute{e}t}(\bar X, \mathbb Q/\mathbb Z(2))^{G_k}] \stackrel{\psi}{\to} H^1(k, H^2_{\acute{e}t}(\bar X, \mathbb Q/\mathbb Z(2)).
$$

\noindent Par composition avec $\psi$, on obtient ainsi l'application $$H^1(X, \mathcal K_2)\otimes \mathbb Q/\mathbb Z\stackrel{\phi}{\to} H^1(k, H^2_{\acute{e}t}(\bar X, \mathbb Q/\mathbb Z(2))).$$

\noindent Il suffit donc de borner les exposants des groupes $\mathrm{coker}\,(\phi)$ et $F_1(X)\cap NH^3_{\acute{e}t}(X, \mathbb Q/\mathbb Z(2))$:

\begin{enumerate}
\item 
 
D'apr\`es \cite[Th\'eor\`eme 7.3]{CT-cycles} et sa preuve, on dispose d'une inclusion  
\begin{equation}\label{cokerphi}
\mathrm{coker}\,(\phi)\subset \mathrm{coker}\,[H^1(k, \NS(\bar X)\otimes \Q/\Z(1))\to H^1(k, H^2_{\acute{e}t}(\bar X, \Q/\Z(2)))],
\end{equation}
o\`u la fl\`eche de droite est induite par la fl\`eche
$$\NS(\bar X)\otimes \Q/\Z=\Pic(\bar X)\otimes \Q/\Z\to H^2_{\acute{e}t}(\bar X, \Q/\Z(1)).$$
de la proposition \ref{prch2}(4). L'exposant du groupe de droite dans (\ref{cokerphi}) est donc born\'e par l'exposant du groupe $\oplus H^3_{\acute{e}t}(\bar X, \mathbb Z_{\ell}(1))_{tors}$. Ainsi l'exposant du groupe $\mathrm{coker}\,(\phi)$ est born\'e  par la constante $n_{31}(X)$.

\item 
Pour le groupe $F_1(X)$, on utilise la courbe $C$ d\'efinie au d\'ebut de la preuve. Soit $F_1(C)$ la partie de la filtration sur $H^3_{\acute{e}t}(C, \Q/\Z(2))$ d\'efinie de la m\^eme mani\`ere que $F_1(X)$.  Par fonctorialit\'e, tout \'el\'ement du groupe $F_1(X)\cap NH^3_{\acute{e}t}(X, \Q/\Z(2))$ se restreint \`a un \'element de $F_1(C)$ dans $F_1(C)\cap NH^3_{\acute{e}t}(C, \Q/\Z(2))$.
On a une suite exacte
$$H^3(k, \mathbb Q/\mathbb Z(2))\to F_1(*)\to H^2(k, H^1_{\acute{e}t}(\bar *, \mathbb Q/\mathbb Z(2)))$$
pour $*=C,X$.

\noindent Dans le cas des courbes,  l'hypoth\`ese $C(k)\neq \emptyset$ implique que le noyau de la fl\`eche naturelle
$$H^2(k, H^1_{\acute{e}t}(\bar C, \mathbb Q/\mathbb Z(2)))\to H^2(k, H^1_{\acute{e}t}(\bar k(C), \mathbb Q/\mathbb Z(2)))$$
est nul \cite[Thm. 3.7 et sa preuve]{R3}. Ainsi, l'image du groupe $F_1(X)\cap NH^3_{\acute{e}t}(X, \Q/\Z(2))$ dans  $H^2(k, H^1_{\acute{e}t}(\bar C, \mathbb Q/\mathbb Z(2)))$ est nulle.
En utilisant l'\'egalit\'e $H^3(k, \mathbb Q/\mathbb Z(2))=(\mathbb Z/2)^r$ o\`u $r$ est le nombre de places r\'eelles de $k$, inf\'erieur au degr\'e de $k$, on se ram\`ene \`a borner l'exposant du noyau de la fl\`eche  de restriction 
$$H^2(k, H^1_{\acute{e}t}(\bar X, \mathbb Q/\mathbb Z(2)))\to H^2(k, H^1_{\acute{e}t}(\bar C, \mathbb Q/\mathbb Z(2))).$$ 
C'est l'objet du Lemme \ref{leminv2} ci-dessous, qui conclut la preuve.  
\end{enumerate}
\end{proof}

\begin{lemma}\label{leminv2}
Soit $X$ une vari\'et\'e projective lisse g\'eom\'etriquement int\`egre d\'efinie sur un corps de nombres $k$ et soit $C\subset X$ une courbe, section hyperplane lisse de $X$, telle que $C(k)\neq \emptyset$. Supposons que l'on a un morphisme $\Pic^0_{C/k}\to \Pic^0_{X/k}$ tel que la fl\`eche compos\'ee
$\Pic^0_{X/k}\to \Pic^0_{C/k}\to \Pic^0_{X/k}$ est une isog\'enie dont le noyau est d'exposant  $d$.
Alors le groupe 
$$\Ker [H^2(k, H^1_{\acute{e}t}(\bar X, \mathbb Q/\mathbb Z(2)))\to H^2(k, H^1_{\acute{e}t}(\bar C, \mathbb Q/\mathbb Z(2)))]$$ est d'exposant fini, et cet exposant ne d\'epend que de $d$, de $n_{22}(X)$ et de $n_{22}(C)$.
\end{lemma}
\begin{proof} La Proposition \ref{prch2}(3) appliqu\'ee \`a $X$ et \`a $C$ donne le diagramme commutatif suivant:

$$
\xymatrix{
0\ar[r]& \varinjlim_n (\Pic^0_{X/k}(\bar k)[n]\otimes \mu_n) \ar[r]\ar[d]& H^1_{\acute{e}t}(\bar X, \Q/\Z(2))\ar[r]\ar[d]&  F_X \ar[r]\ar[d]& 0\\
0\ar[r]& \varinjlim_n (\Pic^0_{C/k}(\bar k)[n]\otimes \mu_n) \ar[r]& H^1_{\acute{e}t}(\bar C, \Q/\Z(2))\ar[r]&  F_C \ar[r]& 0.
}
$$

Prenant les suites exactes longues de cohomologie, on trouve le diagramme commutatif suivant :

$$
\xymatrix{
H^1(k,F_X)\ar[r]\ar[d]& H^2(k,\varinjlim_n (\Pic^0_{X/k}(\bar k)[n]\otimes \mu_n)) \ar[r]\ar^{\iota}[d]& H^2(k,H^1_{\acute{e}t}(\bar X, \Q/\Z(2)))\ar[r]\ar[d]&  H^2(k,F_X)\ar[d]\\
H^1(k,F_C)\ar[r]& H^2(k,\varinjlim_n (\Pic^0_{C/k}(\bar k)[n]\otimes \mu_n)) \ar[r]& H^2(k,H^1_{\acute{e}t}(\bar C, \Q/\Z(2)))\ar[r]&  H^2(k,F_C).
}
$$

Par chasse au diagramme on obtient que  l'exposant du groupe $$\Ker [H^2(k, H^1_{\acute{e}t}(\bar X, \mathbb Q/\mathbb Z(2)))\to H^2(k, H^1_{\acute{e}t}(\bar C, \mathbb Q/\mathbb Z(2)))]$$ de l'\'enonc\'e est born\'e par les exposants des groupes $H^2(k,F_X)$, $H^2(k,F_C)$, et $\Ker \iota$.  L'exposant du premier groupe est born\'e par $n_{22}(X)$ puisque le groupe fini $F_X$ est de m\^eme cardinal que $\oplus_{\ell}H^2_{\acute{e}t}( \bar X, \mathbb Z_{\ell}(2))\{\ell\}$). L'exposant du deuxi\`eme est born\'e par $n_{22}(C)$ par la m\^eme raison.

Pour borner l'exposant du troisi\`eme groupe, soit $K$ le noyau de la surjection compos\'ee
$$\varinjlim_n (\Pic^0_{X/k}(\bar k)[n]\otimes \mu_n)\to \varinjlim_n (\Pic^0_{C/k}(\bar k)[n]\otimes \mu_n)\to \varinjlim_n (\Pic^0_{X/k}(\bar k)[n]\otimes \mu_n).$$
Alors $K$ est tu\'e par $d$ par hypoth\`ese, ce qui implique que le noyau de la compos\'ee $$H^2(k,\varinjlim_n (\Pic^0_{X/k}(\bar k)[n]\otimes \mu_n))\to H^2(k,\varinjlim_n (\Pic^0_{C/k}(\bar k)[n]\otimes \mu_n))\to H^2(k,\varinjlim_n (\Pic^0_{X/k}(\bar k)[n]\otimes \mu_n))$$ est tu\'e par $d$. Cela termine la preuve du lemme.

\end{proof}

\subsubsection{Bornes uniformes pour la partie arithm\'etique}

L'\'enonc\'e suivant donne une borne uniforme pour l'exposant de la partie arithm\'etique $Ker[ CH^2(\mathcal X_s)\to CH^2({\mathcal X}_{\bar s})]$ du groupe $CH^2(\mathcal X_s)_{tors}$.

\begin{corollary}\label{exposant2}
Soit $k$ un corps de nombres. Soit $S$ un sch\'ema s\'epar\'e de type fini sur $k$ et soit $\pi:\mathcal X\to S$ un morphisme projectif lisse, \`a fibres g\'eometriquement int\`egres. Supposons  que $H^2(\mathcal X_s,\mathcal O_{\mathcal X_s})=0$ pour tout point $s\in S$.
Soit $d$ un entier positif. Il existe une constante $N=N(\pi,d)$, qui ne d\'epend que de $d$ et de $\pi$ et qui v\'erifie la propri\'et\'e suivante: si $K$ est une extension de $k$ de degr\'e au plus $d$ et $s$ un $K$-point de $S$, alors:
$$N\cdot Ker[ CH^2(\mathcal X_s)\to CH^2(\mathcal X_{\bar s})]=0.$$ 
\end{corollary}

\begin{proof}
Il suffit d'appliquer le th\'eor\`eme \ref{exposant2-k} \`a la vari\'et\'e $\mathcal X_s$ sur le corps de nombres $K$ de degr\'e au plus $d[k:\mathbb Q]$ sur $\mathbb Q$ et de remarquer que les invariants \ref{subsubsection:invariants}  sont born\'es ind\'ependamment de $s$ gr\^ace aux Lemmes \ref{idansf}, \ref{iNSdansf} et \ref{dhdansf}.
\end{proof}

\subsection{Exposant uniforme}
\

\noindent {\it Preuve du Th\'eor\`eme \ref{theorem:exposant}.}
Soient $\pi:\mathcal X\to S, s\in S,\ell$ comme dans l'\'enonc\'e. D'apr\`es le corollaire \ref{exposant2}, on a un entier $N_1=N_1(\pi,d)$ qui borne l'exposant du groupe $\Ker[ CH^2(\mathcal X_s)\to CH^2(\bar{\mathcal X}_s)]$. Puisque $S$ est une courbe, on peut appliquer le th\'eor\`eme \ref{theorem:exponentgeo} qui donne une borne $N_2=N_2(\pi,d,\ell)$ pour l'exposant du groupe $|\mathrm{Im}(CH^2(\mathcal X_s)\lra CH^2(\mathcal X_{\overline s}))\{\ell\}|$. Il suffit donc de prendre $N=N_1N_2$.

\

Dans le cas o\`u le sch\'ema de base $S$ est de dimension sup\'erieure, on peut trouver une borne pour l'exposant de la torsion, qui d\'epend aussi du nombre de places \`a l'infini d'un mod\`ele entier: 

\begin{theorem}\label{theorem:exponent2}
Soit $\mathcal S$ un sch\'ema s\'epar\'e de type fini sur $\mathbb Z$ et soit $\pi:\mathcal X\to \mathcal S$ un morphisme projectif lisse, \`a fibres g\'eometriquement int\`egres. Supposons  que  $H^2(\mathcal X_s,\mathcal O_{\mathcal X_s})=0$ pour tout point $s\in \mathcal S$.
Soient $d,r>0$ des entiers. Il existe une constante $N=N(\pi,d,r)$ qui v\'erifie la propri\'et\'e suivante: 
  soit $k$ un corps de nombres de degr\'e $[k:\mathbb Q]\leq d$ et soit $s$ un $k$-point de $\mathcal S$; si le morphisme 
$$s : \Spec k\lra \mathcal S$$
s'étend en un morphisme 
$$U\lra \mathcal S$$
au-dessus d'un ouvert $U\subset \Spec \mathcal O_k$ tel que  $U$ admet au plus $r$ places \`a l'infini, alors  
$$NCH^2(\mathcal X_s)_{tors}=0$$
\end{theorem}
\begin{proof}
Il suffit d'appliquer le Corollaire \ref{exposant2} et la Proposition \ref{geomgen}.
\end{proof}

\section{Rappels sur la méthode de Saito-Somekawa-Colliot-Thélène-Raskind}

Soit $X$ une vari\'et\'e projective lisse g\'eom\'etriquement int\`egre sur un corps de nombres $k$, telle que $H^2(X, \mathcal O_X)=0$. Pour montrer la finitude du groupe $CH^2(X)_{tors}$ on proc\`ede en deux \'etapes: on trouve une borne sur l'exposant (ce qui est fait dans la section pr\'ec\'edente), et on trouve une borne sur le sous-groupe de $n$-torsion $CH^2(X)[n]$ pour $n$ fix\'e. Pour la deuxi\`eme \'etape, J.-L. Colliot-Th\'el\`ene et W. Raskind \cite{CTR} utilisent la m\'ethode de {\it localisation} qui permet de relever les \'el\'ements du groupe  $CH^2(X)[n]$ en des \'el\'ements du groupe $CH^2(\mathcal X_U)[N]$ o\`u $\mathcal X_U$ est un mod\`ele de $X$ sur un ouvert $U\subset \Spec \Ok$ pour un certain entier $N$.  

Dans cette section, on reprend certaines variantes des arguments pr\'ec\'edents en les adaptant de façon à pouvoir les utiliser de manière effective dans la section suivante. L'enjeu principal est la dépendance des constantes qui apparaissent dans la torsion dans $CH^2$ en le groupe fondamental étale de la base $U$. Cette dépendance rend délicat le fait de remplacer $U$ par un revêtement ramifié fini qui, même lorsque l'on en borne le degré, peut avoir un groupe fondamental étale arbitrairement grand. C'est cette raison qui rend nécessaire d'adapter certains arguments.

\subsection{Une suite de localisation}

La suite exacte suivante est le point de départ du contrôle de la torsion dans $CH^2$.

\begin{lemma}\label{sloc}  (\cite[Proposition 1.2]{R2}; voir  \cite[suite $(\mathcal L)$ p.231]{CTR})
Soit $X$ une vari\'et\'e projective et lisse, g\'eom\'etriquement int\`egre, d\'efinie sur un corps de nombres $k$ et soit $\mathcal X_U \to U$ un mod\`ele projectif et lisse de $X$ sur un ouvert $U\subset \Spec \Ok$.  On a une suite exacte
\begin{equation}\label{suiteloc}
H^1(X,\mathcal K_2)\to \bigoplus\limits_{s\in U^{(1)}} \mathrm{Pic}(\mathcal X_s)\to CH^2(\mathcal X_U)\to CH^2(X)\to 0.
\end{equation}

\end{lemma}

On s'int\'eresse \`a comprendre le conoyau de l'application de gauche de la suite ci-dessus. La première étape est le lemme suivant, adapté de \cite[Lemme 3.2]{CTR}.

\begin{lemma}\label{lemme3.2}
Soit $X$ une vari\'et\'e projective, lisse, g\'eom\'etriquement int\`egre d\'efinie sur un corps de nombres $k$ et soit $\mathcal X_U \to U$ un mod\`ele projectif et lisse de $X$ sur un ouvert $U$ de l'anneau des entiers $\Ok$ de $k$.  On suppose que $H^2(\mathcal X_s,\mathcal O_{\mathcal X_s})=0$ pour tout point $s\in U$. 

Suposons qu'il existe  un entier strictement positif $N$ tel que pour tout point $s\in  U$ et pour tout point g\'eom\'etrique $\bar s$ au-dessus de $s$, le conoyau de la flèche naturelle $\Pic (\mathcal X_U)\to \NS(\mathcal X_{\bar s})$ est tué par $N$. Alors le conoyau de l'application compos\'ee
$$H^1(X, \mathcal K_2)\to \bigoplus\limits_{s\in U^{(1)}} \mathrm{Pic}(\mathcal X_s)\to \bigoplus\limits_{s\in U^{(1)}} \mathrm {NS}(\mathcal X_{\bar s})$$
est d'exposant fini, tu\'e par l'entier 
$$N\cdot\# Cl(U).$$
\end{lemma}

\begin{proof}

%Si $d=1$, le groupe de classe de $U$ est trivial et $X(k)$ est non vide, alors l'\'enonc\'e est \cite[Lemme 3.2]{CTR}. On rappelle ici les arguments de \cite{CTR} ici en les étendant au cas général.

\noindent Soit $\{l_s\}_{s\in U^{(1)}}\in \bigoplus\limits_{s\in U^{(1)}} \mathrm {NS}(\mathcal X_{\bar s})$. 
D'apr\`es l'hypoth\`ese, la classe $N l_s$ se rel\`eve en un \'el\`ement $L_s\in \mathrm{Pic}(\mathcal X_U)$.

\noindent On \'ecrit $U=\mathrm{Spec}\,B$. Soit $d=\#Cl(U)$. Pour tout id\'eal premier $s\in \mathrm{Spec}\,B$ l'idéal $s^d$ est principal et l'on peut donc \'ecrire $s^d=(\pi_s)$ pour un certain $\pi_s\in B$. Soit $v:k^*\to \mathbb Z$ la valuation associ\'ee \`a $s$. On a $v(\pi_s)=d$ et $v_{s'}(\pi_s)=0$ pour $s'\neq s$. 

\noindent Comme dans \cite[Lemme 3.2]{CTR} on trouve que $\{l_s\}_{s\in U^{(1)}}^{\otimes Nd}\in \bigoplus\limits_{s\in U^{(1)}} \mathrm{NS}(\mathcal X_{\bar s})$ est l'image de $\sum L_s\otimes \pi_s$ par la fl\`eche compos\'ee
$$\Pic(\mathcal X_U)\otimes k^*\to\mathrm {Pic}(X)\otimes k^*\to H^1(X, \mathcal K_2)\to \bigoplus\limits_{s\in U^{(1)}} \mathrm{Pic}(\mathcal X_s)\to \bigoplus\limits_{s\in U^{(1)}} \mathrm {NS}(\mathcal X_{\bar s}).$$

\end{proof}

Pour comprendre la partie qui vient du groupe $\bigoplus\limits_{s\in U^{(1)}} \mathrm{Pic^0}(\mathcal X_s)$ dans la suite (\ref{suiteloc}), on va utiliser les techniques d\'evelopp\'ees par M. Somekawa dans \cite{S90}.  On rappelle ces techniques dans le paragraphe suivant pour pouvoir les utiliser de mani\`ere effective.

%En utilisant les suites de localisation, dans la section \ref{sectionexpoue} on obtient alors une borne uniforme pour l'entier $N=N(n)$ tel que tout \'el\'ement du groupe $CH^2(X)[n]$ se rel\`eve en un \'element de $CH^2(\mathcal X_U)[N]$. Ce dernier groupe est un groupe fini, sous-quotient du groupe de cohomologie \'etale  $H^3_{\acute{e}t}(\mathcal X_U, \mu_N^{\otimes 2})$ (cf. \cite[(3.11), Th\'eor\`eme 6.2 ]{CT-cycles}). Pour terminer la preuve, dans la section \ref{sectionU} on voit que pour $U$ fix\'e, on n'a qu'un nombre fini de valeurs possibles pour la taille du groupe $H^3_{\acute{e}t}(\mathcal X_U, \mu_N^{\otimes 2})$.

\subsection{Bornes uniformes dans la suite exacte de Bloch-Kato-Saito-Somekawa}

%Dans ce paragraphe on rappelle une suite exacte de r\'eciprocit\'e sous la forme d\'emontr\'ee par Somekawa \cite[Theorem 4.1]{S90} (ce qui utilise aussi des travaux de Bloch, Kato et Saito). 
%

\subsubsection{Le groupe $K(k, A, \mathbb G_m)$.} Soit $k$ un corps et soit $A$ une vari\'et\'e ab\'elienne sur $k$. Somekawa \cite{S90} a d\'efini le groupe $K(k, A, \mathbb G_m)$ comme un quotient du groupe $$\bigoplus\limits_{L/k\,\text{fini}} A(L)\otimes_{\mathbb Z} L^*$$ par des relations de deux types: formule de projection, et r\'eciprocit\'e \`a la Weil. Dans ce texte, on n'aura pas besoin de pr\'eciser ces relations. Notons par ailleurs que la construction de Somekawa est plus g\'en\'erale : elle est en particulier valable pour une famille de vari\'et\'es semi-ab\'eliennes sur $k$.

Si $n$ est un entier inversible dans $k$, on dispose d'une application naturelle \cite[Proposition 1.5]{S90}:
\begin{equation}\label{flechec}
K(k, A, \mathbb G_m)/n\to H^2_{\acute{e}t}(k, A(1)[n]).
\end{equation}

\subsubsection{R\'esidus dans le cas local.} Supposons maintenant $k$ local de caract\'eristique $0$ et soit $\mathcal A$ un sch\'ema ab\'elien sur l'anneau des entiers $\mathcal O_k$ de $k$. Soit $A=\mathcal A_k$ la fibre g\'en\'erique de $\mathcal A/\mathcal O_k$. Soit $\kappa$ le corps r\'esiduel de $\mathcal O_k$ et soit $\mathcal A_{\kappa}$ la fibre sp\'eciale de $\mathcal A$. 
Si $L/k$ est une extension finie de $k$ de corps r\'esiduel $\kappa_L$, on dispose d'une fl\`eche de bord 
$A(L)\otimes L^*\to \mathcal A_{\kappa}(\kappa)$ obtenue par la composition 
\begin{equation}\label{fbord}
A(L)\otimes L^*\to \mathcal A_{\kappa_L}(\kappa_L) \to \mathcal A_{\kappa}(\kappa).
\end{equation}
La premi\`ere fl\`eche envoie $x\otimes \alpha\in \mathcal A(L)\otimes L^*=\mathcal A(\mathcal O_L)\otimes L^*$ sur $v_L(\alpha)\bar x$ o\`u $\bar x$ est l'image de $x$ dans $\mathcal A_{\kappa_L}(\kappa_L)$ et o\`u $v_L$ est la valuation de $L$. La deuxi\`eme fl\`eche est induite par la norme.  La fl\`eche  (\ref{fbord}) ci-dessus passe au quotient par les relations et induit une fl\`eche \cite[p.114]{S90}
$$\partial: K(k, A, \mathbb G_m)\to \mathcal A_{\kappa}(\kappa).$$ 

Soit $T(A)=\varprojlim_n A(\bar k)[n]$ le module de Tate de $A$, et soit $T_{\ell}(A)=\varprojlim_r A(\bar k)[\ell^r]$ o\`u $\ell$ est un nombre premier. Soit $G_k=Gal(\bar k/k)$ le groupe de Galois absolu de $k$ et soit $T(A)_{G_k}$ le module des coinvariants. Puisque $\mathcal A$ est lisse sur $\mathcal O_k$, on dispose d'une surjection $\alpha: T(A)_{G_k}\to \mathcal A_{\kappa}(\kappa)$. 

\

 \noindent La compos\'ee de la fl\`eche  (\ref{flechec}) avec la dualit\'e locale ({\it loc .cit.}) $H^2_{\acute{e}t}(k, A(1)[n])\simeq A(\bar k)[n]_{G_k}\;$ donne une fl\`eche 
  $K(k, A, \mathbb G_m)/n\to A(\bar k)[n]_{G_k}$. En passant \`a la limite on obtient ainsi une application  
 \begin{equation}\label{flechec2}
 c(k):K(k, A, \mathbb G_m)\to T(A)_{G_k}.
 \end{equation}

 Si $\mathcal A$ est un sch\'ema ab\'elien comme ci-dessus, on a le diagramme commutatif suivant:

\begin{equation}\label{dS}
\xymatrix{
\bigoplus\limits_{L/k\,\text{fini}} A(L)\otimes L^*\ar@{->>}[r]^{\pi}\ar[d]&  K(k, A, \mathbb G_m)\ar[r]^{c(k)}\ar[d]^{\partial} &   T(A)_{G_k} \ar@{->>}[d]^{\alpha}\\
\mathcal A_{\kappa}(\kappa) \ar@{=}[r]& \mathcal A_{\kappa}(\kappa)  \ar@{=}[r]  & \mathcal A_{\kappa}(\kappa),
}
\end{equation}
o\`u les applications $\pi$,  $\alpha$, $\partial$, et $c(k)$  sont surjectives (voir \cite[Theorem 3.3]{S90} pour l'application $c(k)$).

\

\subsubsection{Cas global}
Supposons maintenant que $k$ est un corps de nombres. Soit $S$ un ensemble fini de places de $k$ qui contient les places archim\'ediennes et les places de mauvaise r\'eduction de $A$. Comme ci-dessus, on \'ecrit $T(A)_{G_k}$  pour le module de Tate de $A$.  Pour toute place $v$ de $k$ on note $k_v$ le compl\'et\'e de $k$,  $A_v=A_{k_v}$, $G_v=G_{k_v}$ et $T(A_v)_{G_v}$ le module de Tate de $A_v$. 

D'apr\`es un r\'esultat de Katz et Lang, $T(A)_{G_k}$ est fini. Plus pr\'ecis\'ement, on a la borne suivante :
\begin{proposition}[\cite{KL82}, Theorem 1(bis), Theorem 1(ter)]\label{KL}
Soit $v$ une place finie de $k$, de corps r\'esiduel $\kappa(v)$, en laquelle $A$ a  bonne r\'eduction $A_{\kappa(v)}$. Alors 
\begin{enumerate}
\item
on a une surjection 
$$T_{\ell}(A_{\kappa(v)})_{G_{\kappa(v)}} \twoheadrightarrow T_{\ell}(A)_{G_k}$$ pour tout premier $\ell\neq car(\kappa(v))$;
\item  on a l'\'egalit\'e de cardinaux des groupes fini
$$\#T(A_{\kappa(v)})_{G_{\kappa(v)}}=\# A_{\kappa(v)}(\kappa(v)).$$
\end{enumerate}
\end{proposition}

\begin{corollary}\label{corKL}
Supposons que $A$ a bonne r\'eduction en deux places  $v_1, v_2$  de caract\'eristiques r\'esiduelles diff\'erentes. Alors le groupe $T(A)_{G_k}$ est fini, de cardinal au plus 
$$\#T(A)_{G_k}\leq \# A_{\kappa(v_1)}(\kappa(v_1))\#A_{\kappa(v_2)}(\kappa(v_2)).$$
\end{corollary}

\begin{proof}
Pour $i=1, 2$, soit $p_i$ la caract\'eristique r\'esiduelle de $v_i$. La proposition pr\'ec\'edente montre que la partie premi\`ere \`a $p_i$ de $T(A)_{G_k}$ est born\'ee sup\'erieurement par  $\# A_{\kappa(v_i)}(\kappa(v_i))$, ce qui conclut.
\end{proof}

Somekawa a \'etabli la suite de r\'eciprocit\'e suivante.

\begin{theorem}[\cite{S90}, Theorem 4.1]\label{theoSomekawa}
Soit $n$ un entier divisible par l'ordre du groupe fini $T(A)_G$. Alors on a une suite exacte
$$ K(k, A, \mathbb G_m)\stackrel{\theta}{\to} \prod\limits_{v\notin S} T(A_v)_{G_v} \oplus \prod_{v\in S} K(k_v, A_v, \mathbb G_m)/n \stackrel{R}{\to} T(A)_G\to 0,$$ o\`u les fl\`eches $\theta, R$ sont induites par les applications $c(k_v)$ de (\ref{flechec2}).
\end{theorem}

Dans la suite, on aura besoin du corollaire suivant (cf. \cite[Th\'eor\`eme 2.1]{CTR}).

\begin{corollary}\label{coroSomekawa}
Soit $U\subset \Spec \Ok$ un ouvert non vide. Soit $\mathcal A$ un $U$-sch\'ema ab\'elien. Soit $A$ la fibre g\'en\'erique de $\mathcal A$. Soit
$$\gamma: \bigoplus\limits_{L/k\,\text{fini}}  A(L)\otimes_{\mathbb Z} L^* \to  \bigoplus\limits_{s\in U^{(1)}} \mathcal A_s(\kappa(s))$$ la fl\`eche induite par les applications (\ref{fbord}).
Alors le conoyau de $\gamma$ est fini, born\'e sup\'erieurement par $\# \mathcal A_{s_1}(\kappa(s_1))\# \mathcal A_{s_2}(\kappa(s_2)),$ o\`u $s_1,s_2$ sont deux points quelconques de $U$ de caract\'eristiques r\'esiduelles diff\'erentes.
\end{corollary}
\begin{proof}
On applique le Th\'eor\`eme \ref{theoSomekawa} avec  $S$  le compl\'ementaire de l'ensemble des places sur $U$ dans l'ensemble des  places de $k$. 
Ainsi, on a le diagramme commutatif suivant
$$
\xymatrix{
K(k, \mathcal A, \mathbb G_m)\ar[r]^{\theta_0}\ar@{=}[d] & \prod\limits_{v\in U} T(\mathcal A_v)_{G_v}  \ar[r]\ar[d] & \coker(\theta_0)\ar[r]\ar[d]& 0  \\
K(k, \mathcal A, \mathbb G_m) \ar[r]^(0.3){\theta}  & \prod\limits_{v\in U} T(\mathcal A_v)_{G_v} \oplus \prod_{v\in S} K(k_v, \mathcal A_v, \mathbb G_m)/n \ar[r]& T( A)_{G_v}\ar[r] & 0
}
$$
dans lequel la fl\`eche verticale du milieu est injective. Ainsi, $\coker(\theta_0)$ est fini, et on a une injection  
\begin{equation}\label{ginj}
\coker(\theta_0)\hookrightarrow T(A)_{G_k}.
\end{equation}

D'apr\`es (\ref{dS}), l'application $\gamma$ se factorise en  $$\bigoplus\limits_{L/k\,\text{fini}}  A(L)\otimes L^*\twoheadrightarrow K(k, \mathcal A, \mathbb G_m)  \stackrel{\theta_0}\to    \prod\limits_{v\in U} T(\mathcal A_v)_{G_v} \twoheadrightarrow \bigoplus\limits_{v\in U}  \mathcal A(\kappa(v))$$
o\`u la premi\`ere et la derni\`ere fl\`eches sont surjectives. Ainsi 
\begin{equation}\label{gsurj}
\# \coker(\gamma)\leq \# \coker(\theta_0),
\end{equation}
d'o\`u $\# \coker(\gamma)\leq \# T(A)_{G_k}$ par (\ref{ginj}) et (\ref{gsurj}).
Il ne reste qu'\`a appliquer le Corollaire \ref{corKL}.
\end{proof}

\

\section{Relèvements entiers}\label{sectionen}

Le but de cette section est de prouver le Théorème \ref{theorem:torsion} en s'appuyant sur les résultats de la section précédente.

On se donne $\mathcal S$ un sch\'ema intègre, s\'epar\'e, de type fini sur $\mathbb Z$. Soit 
$\pi : \mathcal X\lra \mathcal S$
un morphisme projectif lisse \`a fibres g\'eom\'etriquement int\`egres. 

\subsection{Préliminaires}

Soit $\eta$ le point g\'en\'erique de $\mathcal S$ et soit $\bar{\eta}$ un point g\'eom\'etrique de $\mathcal S$ au-dessus de $\eta$. Soit $K=\kappa(\eta)$ le corps résiduel de $\eta$, i.e., le corps des fonctions de $\mathcal S$. Par fonctorialité du groupe fondamental étale, on dispose d'une application naturelle
$$G_K\to \pi_1(\mathcal S, \bar\eta)$$
du groupe de Galois absolu de $K$ vers le groupe fondamental étale de $\mathcal S$.

\begin{lemma}\label{lemma:unramified}
Soit $\ell$ un nombre premier inversible sur $\mathcal S$. L'action naturelle de $G_K$ sur $\NS(\mathcal X_{\bar\eta})\otimes\Z_\ell$ se factorise par $\pi_1(\mathcal S, \bar\eta)$.
\end{lemma}

\begin{proof}
L'application classe de cycle définit une injection $G_K$-équivariante:
$$\NS(\mathcal X_{\bar\eta})\otimes\Z_\ell\to H^2_{\acute{e}t}(\mathcal X_{\bar\eta}, \Z_\ell(1)),$$
ce qui prouve le résultat car l'action de $G_K$ sur $H^2_{\acute{e}t}(\mathcal X_{\bar\eta}, \Z_\ell(1))$ se factorise par $\pi_1(\mathcal S, \bar\eta)$ car $H^2_{\acute{e}t}(\mathcal X_{\bar\eta}, \Z_\ell(1))$ est la fibre en $\bar\eta$ du faisceau localement constant $R^2\pi_*\Z_\ell(1).$
\end{proof}

\begin{lemma}\label{lemma:triv-NS}
Soit $\ell$ un nombre premier inversible sur $\mathcal S$. Il existe un revêtement
$$p : \mathcal S'\to \mathcal S$$
fini étale connexe tel que, si $\eta'$ est le point générique de $\mathcal S'$ et $\bar\eta'$ est un point géométrique de $\mathcal S'$ au-dessus de $\eta'$, alors l'action du groupe de Galois absolu $G_{\kappa(\eta')}$ sur le groupe $\NS(\mathcal X_{\bar \eta'})\otimes\Z_\ell$ est triviale.
\end{lemma}

\begin{proof}
Le lemme \ref{lemma:unramified} montre que le groupe profini $\pi_1(\mathcal S, \bar\eta)$ agit continûment sur le $\mathbb Z_{\ell}$-module de type fini $\NS(\mathcal X_{\bar\eta})\otimes \mathbb Z_{\ell}$, cette action se factorise donc par un groupe fini, ce qui conclut en prenant pour $\mathcal S'$ le revêtement de $\mathcal S$ correspondant à ce quotient fini du groupe fondamental étale.
\end{proof}

\begin{lemma}\label{lemma:finite-index}
Soit $\eta'$ le point générique de $\mathcal S'$ et $\bar\eta'$ un point géométrique au-dessus de $\eta'$.
Avec les notations du Lemme \ref{lemma:triv-NS}, le conoyau de l'inclusion naturelle 
$$i_{\eta'} : \NS(\mathcal X_{\bar \eta'})^{G_{\kappa(\eta')}}\to \NS(\mathcal X_{\bar \eta'})$$
est de torsion.
\end{lemma}

\begin{proof}
Le noyau de l'application 
$$\NS(\mathcal X_{\bar\eta'})\to \NS(\mathcal X_{\bar\eta'})\otimes\Z_\ell$$
est le sous-groupe de torsion première à $\ell$. Soit $M$ un entier strictement positif qui annule ce noyau. Alors si $\alpha$ est un élément de $\NS(\mathcal X_{\bar \eta'})$ et si $\sigma$ est un élément de $G_{\kappa(\eta')}$, l'image de $\alpha$ dans $\NS(\mathcal X_{\bar\eta'})\otimes\Z_\ell$ est invariante par $\sigma$ par construction, donc on a :
$$M(\sigma(\alpha)-\alpha)=0.$$
Cela montre que $M\alpha$ est invariant sous $G_{\kappa(\eta')}$, ce qui conclut.
%
%On dispose d'une injection n
%La conclusion du Lemme \ref{lemma:triv-NS} montre que 
%
%
%la flèche $i_{\eta'}$ devient un isomorphisme après tensorisation par $\Z_\ell$. En particulier, le conoyau de $i_{\eta'}$ est de torsion. \francois{clarifier le produit tensoriel}. Le groupe $\NS(\mathcal X_{\bar\eta'}$ étant de type fini, cela conclut.

%Soit $p$ un nombre premier. Soit $r_p=r_p(s')$ l'exposant du sous-groupe de torsion $p$-primaire de $\NS(X_{\bar s'})$. On sait que $r_p$ est borné de manière indépendante de $s'$ et qu'il existe un ensemble fini $F$ de nombres premiers, indépendant de $s'$ tels que si $p$ n'est pas dans $F$, alors $r_p=0.$
%
%Soit $\alpha$ un élément de $\NS(X_{\bar s'})$. Supposons que, pour un certain entier positif $x$, $p^x\alpha$ est $G_{\kappa(s')}$-invariant. Soit $\sigma$ un élément de $G_{\kappa(s')}$. Alors 
%$$p^x(\sigma(\alpha)-\alpha)=0$$
%donc
%$$p^{r_p}(\sigma(\alpha)-\alpha)=0,$$
%ce qui signifie que $p^{r_p}\alpha$ est $G_{\kappa(s')}$-invariant. Cela montre que le conoyau de $i_{s'}$ est borné par le produit des $r_p$, qui est fini et borné de manière indépendante de $s'$.
\end{proof}

\begin{lemma}\label{lemma:torsion-sp}
Supposons que le morphisme structurel $\mathcal S\to\Spec\Z$ n'est pas surjectif, i.e. qu'il existe un nombre premier $\ell$ inversible sur $\mathcal S$. Alors il existe un rev\^etement
$$p : \mathcal S'\to \mathcal S$$
fini étale connexe tel que, si $\eta'$ est le point générique de $\mathcal S'$ le conoyau de l'application de spécialisation 
$$\Pic(\mathcal X\times_{\mathcal S}\mathcal S')\to \NS(\mathcal X_{\bar \eta'})$$
est de torsion.
\end{lemma}

\begin{proof}
On choisit $p : \mathcal S'\to \mathcal S$ comme dans le lemme \ref{lemma:finite-index} dont on garde les notations. On note $\mathcal X'=\mathcal X\times_{\mathcal S}\mathcal S'.$

Soit $\bar\eta'$ un point géométrique de $\mathcal S'$ au-dessus de $\eta'$. Soient $L_{1, \bar\eta'}, \ldots, L_{\rho, \bar\eta'}$ des éléments de $\Pic(\mathcal X_{\bar\eta'})$ dont les images $l_i$ dans $\NS(\mathcal X_{\bar\eta'})^{G_{\kappa(\eta')}}$ forment une famille génératrice de $\NS(\mathcal X_{\bar\eta'})^{G_{\kappa(\eta')}}$.

On peut trouver une extension finie Galoisienne $K$ du corps résiduel $\kappa(\eta')$ telle que les $L_{i, \bar\eta'}$ soient tous définis sur $K$. Soit $L_{i, \eta'}\in \Pic(\mathcal X_{\bar\eta})^{G_{\kappa(\eta')}}$ le produit tensoriel des conjugués de $L_{i, \bar\eta'}$ par les éléments du groupe de Galois de $K/\kappa(\eta')$. Le conoyau de la flèche naturelle 
$$\Pic(\mathcal X_{\eta'})\to\Pic(\mathcal X_{\bar\eta})^{G_{\kappa(\eta')}}$$
est tué par l'indice $i_{\eta'}$ de la fibre générique $\mathcal X_{\eta'}$. En particulier, le fibré en droites 
$$L'_{i, \eta'}:=L_{i, \eta'}^{\otimes i_{\eta'}}$$
provient de $\Pic(\mathcal X_{\eta'})$. Soit $L'_i$ un élément de $\Pic(\mathcal X')$ qui s'envoie sur $L'_{i, \eta'}$. 

Soit $r_{\eta'}$ l'application composée
$$\Pic(\mathcal X')\to \Pic(\mathcal X_{\eta'})\to \NS(\mathcal X_{\bar\eta'}).$$
Par hypothèse, l'image $l_i$ de $L_{i, \bar\eta'}$ dans $\NS(\mathcal X_{\bar\eta'})$ est Galois-invariante. Par construction, on a donc 
$$r_{\eta'}(L'_i)=[K:\kappa(\eta')]i_{\eta'} l_i.$$
En particulier, les images des $L'_i$ par $r_{\eta'}$ engendrent un sous-groupe qui contient 
$$[K:\kappa(\eta')] i_{\eta'}\NS(\mathcal X_{\bar\eta'})^{G_{\kappa(\eta')}}.$$
Appliquant le Lemme \ref{lemma:finite-index}, on trouve que le conoyau de 
$$r_{\eta'} : \Pic(\mathcal X')\lra \NS(\mathcal X_{\bar\eta'})$$
est de torsion.
% tué par 
%$$N:=M[K:\kappa(\eta')] i_{\eta'}.$$
\end{proof}

\begin{proposition}\label{proposition:reduc-etale}
Supposons  que  $H^2(\mathcal X_s,\mathcal O_{\mathcal X_s})=0$ pour tout point $s\in \mathcal S$.
Supposons que le morphisme structurel $\mathcal S\to\Spec\Z$ n'est pas surjectif, i.e. qu'il existe un nombre premier $\ell$ inversible sur $\mathcal S$. Alors il existe un rev\^etement
$$p : \mathcal S'\to \mathcal S$$
fini étale connexe, et un entier $N$ tel que pour tout point g\'eom\'etrique $\bar s'$ de $\mathcal S'$, le conoyau de l'application de spécialisation 
$$\Pic(\mathcal X\times_{\mathcal S}\mathcal S')\to \NS(\mathcal X_{\bar s'})$$
est tué par $N$.
\end{proposition}

\begin{proof}
Soit $\bar s'$ un point géométrique de $\mathcal S'$. On dispose des applications de spécialisation
$$r_{s'} : \Pic(\mathcal X\times_{\mathcal S}\mathcal S')\to \NS(\mathcal X_{\bar s'})$$
et
$$r^{\NS}_{s'} : \NS(\mathcal X_{\bar \eta'})\to\NS(\mathcal X_{\bar s'}),$$ o\`u
$r^{\NS}_{s'}$ est surjective. C'est ici que l'on utilise l'hypothèse d'annulation de $H^2(\mathcal X_{s'}, \mathcal O_{X_{s'}})$ via \cite[Corollaire 1 de la Proposition 3 p.11]{GrothendieckGF} qui garantit l'existence de relêvements de fibrés en droites, voir l'argument de \cite[Lemme 3.1]{CTR}. On a un diagramme commutatif
\[
\xymatrix{
\Pic(\mathcal X\times_{\mathcal S}\mathcal S')\ar[r]^{r_{\eta'}}\ar[dr]^{r_{s'}} & \NS(\mathcal X_{\bar \eta'})\ar[d]^{r^{\NS}_{s'}}\\ 
&  \NS(\mathcal X_{\bar s'}).
}
\]
 Il suit immédiatement que le conoyau de $r_{\eta'}$ se surjecte sur le conoyau de $r_{s'}$, ce qui conclut grâce au Lemme \ref{lemma:torsion-sp}.
\end{proof}

\subsection{Preuve du Théorème \ref{theorem:torsion}}\label{sectionexpoue}

\subsubsection{}

L'énoncé qui suit est une variante effective de \cite[Lemme 3.3]{CTR}, dont on va rappeler les arguments.

Soit $X$ une  vari\'et\'e projective, lisse, g\'eom\'etriquement int\`egre, d\'efinie sur un corps de nombres $k$. Soit $\mathcal X_U\to U$ un mod\`ele projectif et lisse de $X$ sur un ouvert $U\subset \Spec \Ok$. Supposons que  $H^2(\mathcal X_s,\mathcal O_{\mathcal X_s})=0$ pour tout point $s\in U$.

Suposons qu'il existe  un entier strictement positif $N$ tel que pour tout point $s\in  U$ et pour tout point g\'eom\'etrique $\bar s$ au-dessus de $s$, le conoyau de la flèche naturelle $\Pic (\mathcal X_U)\to \NS(\mathcal X_{\bar s})$ est tué par $N$.

\begin{proposition}\label{relevement}
Soit $\mathcal J$ le sch\'ema ab\'elien $Pic^0_{\mathcal X_U/U}$.  On fixe  $s_1, s_2\in U$ deux points ferm\'es de $U$ de caract\'eristiques r\'esiduelles diff\'erentes.  Il existe un entier $n'$ qui ne d\'epend que de $N$, $\#Cl(U)$ et des fibres $\mathcal J_{s_1}, \mathcal J_{s_2}$ tel que le noyau de la restriction 
$CH^2(\mathcal X_U)\to CH^2(X)$ est tu\'e par $n'$. En particulier,
 pour tout entier strictement positif $n$, l'image de l'application naturelle  $CH^2(\mathcal X_U)[nn']\to CH^2(X)[nn']$ contient le groupe $CH^2(X)[n]$.
\end{proposition}

\begin{proof}
%Soit $V$ la cl\^oture int\'egrale de $U$ dans $L$, et  soient $s'_1, s'_2$  deux points de $V$ au-dessus de $s_1$,$s_2$ respectivement. On a alors un diagramme commutatif suivant:  
%
%$$
%\xymatrix{
%0\ar[r] & K\ar[r]\ar[d]  & CH^2(\mathcal X_U)\ar[r]\ar[d] & CH^2(X)\ar[r]\ar[d] & 0\\
%0\ar[r] & K_L\ar[r] & CH^2(\mathcal X_{V})\ar[r] & CH^2(X_L)\ar[r] & 0\\
%}
%$$
%
%Soit $d_L=[L:k]$, on a alors que les fibres $\mathcal J_{s'_i}, i=1,2$ sont obtenus par changement de base sur les extensions  [$\kappa(s'_i): \kappa(s_i)$]  de degr\'es au plus $d_L$, et que $d_L$ annulle le groupe $K_U=Ker[CH^2(\mathcal X_U)\to CH^2(\mathcal X_{V})]$ (voir \cite[Theorem 7.22]{Gillet} pour la construction de la norme dans ce cas). 
%%La reference a Gillet: c'est juste l'existence de la norme: E^2=E_2 and le cas equidimentionnel, theorem 7.19 dit que la cohomologie du complexe R_p compute les groupes de Chow, E_2 sont obtenus a partir de ce complexe - pas besoin de comparer avec les faisceaux \mathcal K_j. 
%  La chasse au diagramme montre que l'exposant du groupe $K$ est born\'e par $d_L$ et l'exposant du groupe $K_L$.
%
%

Soit $K=\ker [CH^2(\mathcal X_U)\to CH^2(X)].$ On va prendre pour $n'$ l'exposant de $K$, dont on va montrer qu'il est bien fini et borné par une constante ne dépendant que de $N$, $U$ et des fibres $\mathcal J_{s_1}$ et $\mathcal J_{s_2}$.

De la suite de localisation (\ref{suiteloc}) on d\'eduit:  
$$K=\mathrm{coker} [H^1(X, \mathcal K_2)\to \bigoplus\limits_{s\in U^{(1)}} \mathrm{Pic}(\mathcal X_s)].$$ 

 Soit $J$ la fibre g\'en\'erique de $\mathcal J$. D'apr\`es le  corollaire \ref{coroSomekawa} appliqu\'e  au sch\'ema ab\'elien $\mathcal J$  on a une suite exacte:
 $$ \bigoplus\limits_{F/k\,\text{fini}}  J(F)\otimes_{\mathbb Z} F^* \stackrel{\gamma}{\to}  \bigoplus\limits_{s\in U^{(1)}}\mathcal J_s(\kappa(s))\stackrel{\delta}\to\coker(\gamma)\to 0,$$
o\`u l'exposant du groupe de droite  est born\'e par une constante qui ne d\'epend que de $ \mathcal J_{s_1}$ et $ \mathcal J_{s_2}$.

Remarquons que les groupes $\mathcal J_s(\kappa(s))$ ne sont pas nécessairement des sous-groupes des groupes $\Pic(\mathcal X_s)$ car on n'a pas supposé que $\mathcal X_s$ a un point rationnel.  Toutefois,  $\mathcal J_s(\overline{\kappa(s)})=\Pic^0(\mathcal X_{\bar s})$, la fl\`eche $\Pic(\mathcal X_s) \to \Pic(\mathcal X_{\bar s})$ est injective, et l'inclusion  $\mathcal J_s(\kappa(s))\cap \Pic(\mathcal X_{s})\subset \mathcal J_s(\overline{\kappa(s)})\cap \Pic(\mathcal X_{s})$ est une \'egalit\'e.

On dispose donc d'une suite exacte: 
$$0\to  \bigoplus\limits_{s\in U^{(1)}}\mathcal J_s(\kappa(s))\cap \Pic(\mathcal X_{s})\to \bigoplus\limits_{s\in U^{(1)}} \mathrm{Pic}(\mathcal X_s) \to \bigoplus\limits_{s\in U^{(1)}} \mathrm{NS}(\mathcal X_{\bar s}).$$

Soient $$A=\gamma^{-1}\Big(\bigoplus\limits_{s\in U^{(1)}}\mathcal J_s(\kappa(s))\cap \Pic(\mathcal X_{s})\Big)\mbox{ et }C=\delta\Big(\bigoplus\limits_{s\in U^{(1)}}\mathcal J_s(\kappa(s))\cap \Pic(\mathcal X_{s})\Big),$$
de sorte que $C$ est un sous-groupe du conoyau de $\gamma$.

Comme dans \cite[Lemme 3.3]{CTR}, on a le diagramme commutatif suivant: 

$$
\xymatrix{
& 0\ar[d] & & \\
 A \ar[r]\ar[d] & \bigoplus\limits_{s\in U^{(1)}}\mathcal J_s(\kappa(s))\cap \Pic(\mathcal X_{s})\ar[d]\ar[r] & C\ar[r]\ar[d] & 0 \\
H^1( X,\mathcal K_2)\ar[r]\ar@{->>}[rd] &  \bigoplus\limits_{s\in U^{(1)}} \mathrm{Pic}(\mathcal X_s) \ar[d]\ar[r] & K\ar[r] &0\\
& \bigoplus\limits_{s\in U^{(1)}} \mathrm{NS}(\mathcal X_{\bar s})& & \\       
}
$$     

Une chasse aux diagramme montre que l'exposant de $K$ est borné par le produit de l'exposant de $C$ et de l'exposant du conoyau de la flèche diagonale. Le lemme  \ref{lemme3.2} pour la fl\`eche diagonale, et le  corollaire \ref{coroSomekawa} permettent de conclure.
\end{proof}

\subsubsection{} On peut maintenant donner la preuve du Théorème \ref{theorem:torsion}. On fixe un sch\'ema $\mathcal S$ intègre, s\'epar\'e, de type fini sur $\mathbb Z$ et un morphisme $\pi : \mathcal X\lra \mathcal S$
un morphisme projectif lisse \`a fibres g\'eom\'etriquement int\`egres. Supposons  que  $H^2(\mathcal X_s,\mathcal O_{\mathcal X_s})=0$ pour tout point $s\in \mathcal S$.

Soit $k$ un corps de nombres, $\mathcal O_k$ son anneau des entiers, et soit $U\subset \Spec\Ok$ un ouvert non vide. Soit 
$s : \Spec k\lra \mathcal S$ l'image du point g\'en\'erique d'un $U$-point de $\mathcal S$.  On cherche un entier $N$ ne d\'ependant que de $U$ et de $\pi$ tel que
$$|CH^2(\mathcal X_s)_{tors}|\leq N.$$

Soit $\ell$ un nombre premier. Quitte à remplacer $U$ par $U\times_{\Spec\Z}\Spec\Z[1/\ell]$ et $\mathcal S$ par $\mathcal S\times_{\Spec\Z}\Spec\Z[1/\ell]$, on peut supposer que $\ell$ est inversible sur $U$.

\begin{proposition}\label{proposition:relev-tors}
Il existe un entier $N'$ ne d\'ependant que de $U$ et de $\pi$ tel que, pour tout entier $n$, l'image de l'application naturelle  $CH^2(\mathcal X_U)[nN']\to CH^2(\mathcal X_s)[nN']$ contient le groupe $CH^2(\mathcal X_s)[n].$
\end{proposition}

\begin{proof}
La Proposition \ref{proposition:reduc-etale} montre l'existence d'un morphisme 
$$p : \mathcal S'\to \mathcal S$$
fini étale et d'un entier $N'_1$ ne dépendant que de $\pi$ tel que, pour tout point géométrique $\bar m'$ de $\mathcal S'$, le conoyau de l'application de spécialisation 
$$\Pic(\mathcal X\times_{\mathcal S}\mathcal S')\to \NS(\mathcal X_{\bar m'})$$
est tué par $N'_1$.

Soit $V$ une composante connexe de $U\times_{\mathcal S} \mathcal S'$. On 
a un diagramme commutatif :  

$$
\xymatrix{
0\ar[r] & K\ar[r]\ar[d]  & CH^2(\mathcal X_U)\ar[r]\ar[d] & CH^2(X)\ar[r]\ar[d] & 0\\
0\ar[r] & K_L\ar[r] & CH^2(\mathcal X_{V})\ar[r] & CH^2(X_L)\ar[r] & 0\\
}
$$
o\`u $L$ est le corps des fonctions de $V$.
Le degr\'e $d_L=[L:k]$ est born\'e par le degr\'e de $p$, %que les fibres $\mathcal J_{s'_i}, i=1,2$ sont obtenus par changement de base sur les extensions  [$\kappa(s'_i): \kappa(s_i)$]  de degr\'es au plus $d_L$,
et annule le groupe $\Ker[CH^2(\mathcal X_U)\to CH^2(\mathcal X_{V})]$ (voir \cite[Theorem 7.22]{Gillet} pour la construction de la norme dans ce cas). 
%La reference a Gillet: c'est juste l'existence de la norme: E^2=E_2 and le cas equidimentionnel, theorem 7.19 dit que la cohomologie du complexe R_p compute les groupes de Chow, E_2 sont obtenus a partir de ce complexe - pas besoin de comparer avec les faisceaux \mathcal K_j. 
Une chasse au diagramme montre que l'exposant du groupe $K$ est born\'e par le produit de $d_L$ et de l'exposant du groupe $K_L$.

L'hypothèse de la Proposition \ref{relevement} est satisfaite sur $\mathcal S'$ par construction. L'exposant de $K_L$ est donc borné par une constante $N'_2$ qui ne dépend que de $V$, et $\pi$. En effet, le Lemme \ref{places} garantit que, avec les notations de la Proposition \ref{relevement}, on peut choisir les fibres $\mathcal J_{s'_1}$ et $\mathcal J_{s'_2}$ d'une manière qui ne dépend que de $\mathcal S'$.

Le corps $k$ n'admet qu'un nombre fini d'extensions finies non ramifiées de degré borné sur $U$, voir par exemple \cite[II.6, Lemme 6]{Se}. En particulier, le schéma $V$ ci-dessus ne peut prendre qu'un nombre fini de valeurs, ce qui conclut.
\end{proof}

\begin{proposition}\label{proposition:fin-U}
Soit $n$ un entier strictement positif, inversible sur $U$. Il existe un entier $N''$ ne dépendant que de $n$, $U$ et $\pi$ tel que 
$$|CH^2(\mathcal X_U)[n]|\leq N''.$$
\end{proposition}

\begin{proof}
Comme dans la proposition \ref{prch2}(1), le groupe $CH^2(\mathcal X_U)[n]$ est un sous-quotient du groupe (fini) $H^3_{\acute{e}t}(\mathcal X_U, \mu_{n}^{\otimes 2})$ (voir \cite[(3.11), Th\'eor\`eme 6.2 ]{CT-cycles}).

On va borner la taille du groupe $H^3_{\acute{e}t}(\mathcal X_U, \mu_{n}^{\otimes 2}).$ On utilise la suite spectrale de Leray pour $\pi_U:\mathcal X_U\to U$:
$$H^p_{\acute{e}t}(U, R^q\pi_{U*}\mu_{n}^{\otimes 2})\Rightarrow H^{p+q}_{\acute{e}t}(\mathcal X_U, \mu_{n}^{\otimes 2}).$$

Il suffit donc pour conclure de borner la taille  des groupes  $H^i_{\acute{e}t}(U, R^j\pi_{U*}\mu_{n}^{\otimes 2})$ o\`u  $i+j=3$. Les fibres g\'eom\'etriques du faisceau  $\mathcal F=R^j\pi_{U*}\mu_{n}^{\otimes 2}$ sont isomorphes \`a $H^j_{\acute{e}t}(\mathcal X_{\bar s},\mu_{n}^{\otimes 2})$, elles sont donc constantes dans la famille $\pi:\mathcal X\to \mathcal S$.

Puisque le morphisme $\pi$ est lisse, le faisceau $\mathcal F$ est localement constant. Il existe donc un morphisme fini \'etale $\tau: U'\to U$ tel que le faisceau $\mathcal F_{U'}$ est constant, de fibre $H^j_{\acute{e}t}(\mathcal X_{\bar s},\mu_{n}^{\otimes 2}).$ Soit $F$ le corps des fonctions de $U'$: c'est une extension finie de $k$. 

Notons que le groupe de Galois $G$ du rev\^etement $\tau$ est un sous-groupe du groupe des automorphismes de $H^j_{\acute{e}t}(\mathcal X_{\bar s},\mu_{n}^{\otimes 2})$, on n'a donc qu'un nombre fini de tels groupes, et le cardinal $|G|$ est born\'e. Ainsi l'extension $F/k$ est de degr\'e born\'e. Comme cette extension est non-ramifi\'ee au-dessus de l'ouvert fix\'e $U$ de $\mathcal O_k$ par construction, on n'a qu'un nombre fini de possibilit\'es pour le corps $F$. 

Ainsi seul un nombre fini de groupes de cohomologie $H^i_{\acute{e}t}(U, R^j\pi_{U*}\mu_{n}^{\otimes 2})$ peuvent apparaître comme on le voit en considérant la suite spectrale  $H^p(G, H^q(U', \mathcal F_{U'}))\Rightarrow H^{p+q}(U, \mathcal F)$ o\`u l'on n'a qu'un nombre fini de possibilit\'es pour $G$, $U'$, et les groupes finis $H^q(U', \mathcal F_{U'})$ d'apr\`es ce qui pr\'ec\`ede. On peut donc borner la taille des groupes $H^i_{\acute{e}t}(U, R^j\pi_{U*}\mu_{n}^{\otimes 2})$ par une constante qui ne d\'epend que de $\pi$ et de $U$, d'o\`u le r\'esultat.
\end{proof}

\begin{proof}[Démonstration du Théorème \ref{theorem:torsion}]
Le Théorème \ref{theorem:exponent2} montre qu'il existe un entier $N'''$ ne dépendant que du degré de $k$ sur $\Q$ et de $\pi$ tel que l'exposant du groupe $CH^2(\mathcal X_s)_{tors}$ est borné par $N'''$. La Proposition \ref{proposition:relev-tors} dont on garde les notations montre que l'application naturelle $CH^2(\mathcal X_U)[N'N''']\to CH^2(\mathcal X_s)_{tors}$ est surjective. C'est encore le cas si l'on remplace $U$ par $U\times_{\Spec\Z}\Spec\Z[1/N'N''']$.

Appliquer la Proposition \ref{proposition:fin-U} à $U\times_{\Spec\Z}\Spec\Z[1/N'N''']$ et $n=N'N'''$ conclut la preuve.

\end{proof}

%\subsection{La conjecture de Tate enti\`ere pour les diviseurs}
%
%Soit $X$ une vari\'et\'e projective lisse, g\'eom\'etriquement connexe, sur un corps de nombres $k$. Soit $\ell$ un nombre premier. On s'int\'eresse \`a la torsion $\ell$-primaire dans le conoyau de l'application classes de cycles
%$$\Pic(X)\to H^2(\bar X, \Z_\ell(1))^G.$$
%
%\subsubsection{}
%
%Il suffit de regarder 
%$$H^2(X, \Z_\ell(1))\to H^2(\bar X, \Z_\ell(1)).$$
%
%\subsubsection{}
%
%\begin{lemma}
%Il existe un entier $N$ ne d\'ependant que des groupes de cohomologie \'etale de $\bar X$ et du degr\'e d'une section hyperplane $h$ tel que
%$$\mathrm{Coker}(H^2(X, \Z_\ell(1))\to H^2(\bar X, \Z_\ell(1)))$$
%est tu\'e par $N$.
%\end{lemma}
%
%\begin{proof}
%Un argument de suite spectrale de Hochschild-Serre montre que \, au groupe $H^3(k, \Z_\ell(1))$ pr\`es, $H^2(X, \Z_\ell(1))$ se surjecte sur le noyau de 
%$$d_2 : H^2(\bar X, \Z_\ell(1))\to H^2(k, H^1(\bar X, \Z_\ell(1)).$$
%
%Pour contr\^oler le conoyau qui nous int\'eresse, il suffit donc de trouver un entier $N$ tel que $N d_2=0.$ \'Ecrire la chasse au diagramme qui le montre.
%\end{proof}

%\subsection{Carrés de courbes elliptiques sur $\Q$}

\end{document}